\documentclass[11pt, reqno]{amsart}
\usepackage[marginratio=1:1,height= 600pt, width=440pt,tmargin =90pt]
{geometry}
\usepackage{amsmath}
\usepackage{amsfonts}
\usepackage{amssymb}
\usepackage[all]{xy}           
\usepackage{bm}
\usepackage{bbding}
\usepackage{txfonts}
\usepackage{amscd}
\usepackage[english]{babel}
\usepackage{xspace}
\usepackage[shortlabels]{enumitem}
\usepackage{ifpdf}

\ifpdf
 \usepackage[colorlinks,final,backref=page,hyperindex]{hyperref}
\else
 \usepackage[colorlinks,final,backref=page,hyperindex,hypertex]{hyperref}
\fi
\usepackage{tikz}
\usepackage[active]{srcltx}
\usepackage{xcolor}
\definecolor{rose}{rgb}{1.5,0.0,0.60}

\makeatletter

\newtheorem{thm}{Theorem}[section]

\newtheorem{pro}[thm]{Proposition}
\newtheorem{ex}[thm]{Example}
\theoremstyle{definition}
\newtheorem{rmk}[thm]{Remark}
\newtheorem{defi}[thm]{Definition}

\newcommand{\nc}{\newcommand}
\newcommand{\delete}[1]{}

\nc{\mlabel}[1]{\label{#1}}  
\nc{\mcite}[1]{\cite{#1}}  
\nc{\mref}[1]{\ref{#1}}  
\nc{\mbibitem}[1]{\bibitem{#1}} 

\delete{
\nc{\mlabel}[1]{\label{#1}{\hfill \hspace{1cm}{\bf{{\ }\hfill(#1)}}}}
\nc{\mcite}[1]{\cite{#1}{{\em{{\ }(#1)}}}}  
\nc{\mref}[1]{\ref{#1}{{\em{{\ }(#1)}}}}  
\nc{\mbibitem}[1]{\bibitem[\em #1]{#1}} 
}

\newcommand {\emptycomment}[1]{}

\nc{\oprn}{\theta}
\nc{\Oprn}{\Theta}

\nc{\calo}{\mathcal{O}}
\nc{\oop}{$\mathcal{O}$-operator\xspace}
\nc{\oops}{$\mathcal{O}$-operators\xspace}
\nc{\mrho}{{\bm{\varrho}}}
\nc{\emk}{\mathbf{K}}
\nc{\invlim}{\displaystyle{\lim_{\longleftarrow}}\,}
\nc{\ot}{\otimes}

\newcommand{\be }{\begin{equation}}
\newcommand{\ee }{\end{equation}}

\newcommand{\br}[1]{   [ \cdot,    \cdot  ]   }

\nc{\CV}{\mathbf{C}}

\begin{document}

	\title[A cohomological study of modified Rota-Baxter associative algebras with derivations]
{A cohomological study of modified Rota-Baxter associative algebras with derivations}

\author{Imed Basdouri, Sami Benabdelhafidh, Mohamed Amin Sadraoui, Ripan Saha }
\address{University of Gafsa, Faculty of Sciences Gafsa, 2112 Gafsa, Tunisia.}
\email{\bf basdourimed@yahoo.fr}

\address{University of Sfax, Faculty of Sciences of Sfax, BP 1171, 3038 Sfax, Tunisia.}
\email{\bf abdelhafidhsami41@gmail.com}

\address{University of Sfax, Faculty of Sciences of Sfax, BP 1171, 3038 Sfax, Tunisia.}
\email{\bf aminsadrawi@gmail.com}

\address{Department of Mathematics, Raiganj University
	Raiganj, 733134, West Bengal, India.}
\email{\bf ripanjumaths@gmail.com}

\begin{abstract}
This paper presents a cohomological study of modified Rota-Baxter associative algebras in the presence of derivations. The Modified Rota-Baxter operator, which is a modified version and closely related to the classical Rota-Baxter operator, has garnered significant attention due to its applications in various mathematical and physical contexts. In this study, we define a cohomology theory and also investigate a one-parameter formal deformation theory and abelian extensions of modified Rota-Baxter associative algebras under the influence of derivations.
\end{abstract}


\keywords{Associative algebras, derivation, modified Rota-Baxter operator, cohomology, deformation, extension.}
\subjclass[2020]{16B50, 17B38, 16E40, 16S80, 16S70}
\maketitle

\vspace{-1.1cm}

\tableofcontents

\allowdisplaybreaks

\section{Introduction}
The inception of the Rota-Baxter operator traces back to its emergence within the realm of fluctuation theory in Probability \cite{Baxter}. Subsequent advancements were spearheaded by Rota \cite{Rota} and Cartier \cite{Cartier}. Since then, the Rota-Baxter operator has garnered attention across diverse algebraic structures, including associative algebras \cite{das20}, Lie algebras \cite{guo12, GLS,D21}, Pre-Lie algebras \cite{LHB}, and Leibniz algebras \cite{MS22}. 

A linear map $\mathrm{P: A \rightarrow A}$ is said to be a Rota-Baxter operator of weight $\lambda \in \mathbb{K}$ if the map $\mathrm{P}$ satisfies $$\mathrm{\mu(P(a),P(b)) = P(\mu(P(a),b) + \mu(a,P(b))) + \lambda P\mu(a,b)},$$ for all $\mathrm{a,b \in A}$.

In Semenov-Tyan-Shanski's \cite{sem} work, it was noted that given specific circumstances, a Rota-Baxter operator with weight 0 on a Lie algebra corresponds simply to the operator manifestation of the classical Yang-Baxter equation (CYBE). Furthermore, within the same study, the author introduced a closely associated concept known as the modified classical Yang-Baxter equation (modified CYBE), whose solutions are termed modified $r$-matrices. The modified CYBE has been applied significantly in investigating Lax equations, affine geometry on Lie groups, factorization problems in Lie algebras, and compatible Poisson structures. While the modified CYBE can be derived from the CYBE through an appropriate transformation, they serve distinct roles in mathematical physics. Consequently, various researchers have examined them separately. Motivated by such inquiries, the authors in subsequent works considered the associative counterpart of the modified CYBE, termed the modified associative Yang-Baxter equation with weight $\kappa \in \mathbb{K}$ (abbreviated as modified AYBE with weight $\kappa$). A solution to the modified AYBE with weight $\kappa$ is termed a modified Rota-Baxter operator with weight $\kappa$, representing the associative counterpart of modified $r$-matrices.  the modified Rota-Baxter operator on associative algebras receives attention in \cite{Das2}. Notably, Mondal and Saha delve into the cohomology of modified Rota-Baxter Leibniz algebra of weight $\kappa$ \cite{MS24}.

For a given associative algebra $\mathrm{A}$, a linear mapping $\mathrm{R : A \rightarrow A}$ is called a modified Rota-Baxter operator with weight $\kappa \in \mathbb{K}$ if it satisfies the equation $$\mathrm{\mu(R(a),R(b)) = R(\mu(R(a),b) + \mu(a,R(b))) + \kappa \mu(a,b),}$$ for all $\mathrm{a,b \in A}$. An associative algebra equipped with a modified Rota-Baxter operator of weight $\kappa$ is referred to as a modified Rota-Baxter algebra with weight $\kappa$. It has been noted that if $(A,P)$ constitutes a Rota-Baxter algebra with weight $\lambda$, then $\mathrm{(A,R = \lambda \mathrm{id} + 2P)}$ forms a modified Rota-Baxter algebra with weight $\kappa = -\lambda^2$.

The exploration of derivations on algebraic structures finds its roots in the work of Ritt \cite{ritt}, who laid the groundwork for commutative algebras and fields. The notion of differential (commutative) algebra \cite{CGKS} emerges from this exploration. Derivations across various types of algebras unveil critical insights into their structural nuances. For instance, Coll, Gerstenhaber, and Giaquinto \cite{C89} explicitly delineate a deformation formula for algebras, where the Lie algebra of derivations encompasses the unique non-abelian Lie algebra of dimension two. Amitsur \cite{AS57, AS82} delves into the study of derivations of central simple algebras. Derivations find utility in constructing homotopy Lie algebras \cite{V05} and play a pivotal role in differential Galois theory \cite{M94}. A recent exploratory angle is presented in \cite{loday10}, where algebras with derivations are scrutinized from an operadic perspective. Moreover, Lie algebras with derivations, referred to as LieDer pairs, undergo examination from a cohomological standpoint in \cite{T19, QL}, with extensions and deformations of LieDer pairs being subjects of further consideration. Also derivations are useful in the construction of InvDer algebraic structures in \cite{Bas}

Parallel to the Analytical Deformation Theory of the 1960s, Murray Gerstenhaber \cite{G63, G64} ventured into the formal deformation theory of associative algebras. This endeavor necessitates an apt cohomology, termed deformation cohomology, to oversee the deformations under scrutiny. Gerstenhaber's seminal work showcases how Hochschild cohomology governs the one-parameter formal deformation of associative algebras. In a similar vein, Nijenhuis and Richardson delve into the formal deformation theory for Lie algebras \cite{NR66}. 

This paper undertakes an exploration of modified Rota-Baxter associative algebras with derivations. A modified Rota-Baxter AssDer pair of weight $\kappa$ comprises a modified Rota-Baxter associative algebra of weight $\kappa$, equipped with a derivation satisfying the compatibility condition. Initially, we introduce the concept of modified Rota-Baxter associative algebras with derivations. We investigate the cohomologies of modified Rota-Baxter AssDer pairs in an arbitrary bimodule. Additionally, we examine formal one-parameter deformations of modified Rota-Baxter AssDer pairs. In particular, we demonstrate that the infinitesimal of a formal one-parameter deformation of a modified Rota-Baxter AssDer pair forms a $\mathrm{2}$-cocycle in the cohomology complex with coefficients in itself. Furthermore, the vanishing of the second cohomology implies the rigidity of the modified Rota-Baxter AssDer pair. As an additional application of our cohomology, we explore abelian extensions of modified Rota-Baxter AssDer pairs and establish their classification by the second cohomology groups.

\section{Modified Rota-Baxter AssDer pairs}
\def\theequation{\arabic{section}.\arabic{equation}}
\setcounter{equation} {0}
In this section, we consider the concept of modified Rota-Baxter AssDer pairs and their bimodules. We begin with some basic definitions related to modified Rota-Baxter associative algebras.\\
A modified Rota-Baxter associative algebra of weight $\kappa$ is an associative algebra $(\mathcal{A}=(\mathrm{A,\mu}))$ equipped with a linear map $\mathrm{R:A\rightarrow A}$ such that
\begin{equation*}
	\mathrm{\mu(R(a),R(b))=R\big(\mu(R(a),b)+\mu(a,R(b))\big)+\kappa \mu(a,b)},\quad \forall\mathrm{ a,b\in A}.
\end{equation*}
Denote simply, in this paper, by $(\mathcal{A},\mathrm{R})$.\\
Recall that if $\mathrm{M}$ is a vector space and $\mathrm{\mathbf{l}:A\otimes M\rightarrow M}$, $\mathrm{\mathbf{r}:M\otimes A\rightarrow A}$ be two linear maps satisfying
\begin{equation*}
	\mathrm{\mathbf{l}(\mu(a,b),m)=\mathbf{l}(a,\mathbf{l}(b,m))},\quad \mathrm{\mathbf{l}(a,\mathbf{r}(m,b))=\mathbf{r}(\mathbf{l}(a,m),b)},\quad \mathrm{\mathbf{r}(m,\mu(a,b))=\mathbf{r}(\mathbf{r}(m,a),b)},
\end{equation*}
for all $\mathrm{a,b\in A} \text{ and } \mathrm{m\in M}$ 
then $\mathrm{\mathcal{M}=(M,\mathbf{l,r})}$ is an $\mathrm{A}$-bimodule.\\
A bimodule over the modified Rota-Baxter associative algebra $(\mathcal{A},\mathrm{R})$ of weight $\kappa$  is a couple $(\mathcal{M},\mathrm{R_M})$ in which $\mathcal{M}$ is an $\mathcal{A}$-bimodule and $\mathrm{R_M}:\mathrm{M\rightarrow M}$ is a linear map such that for $\mathrm{a\in A,m\in M}$,
\begin{eqnarray}
	\mathrm{\mathbf{l}(R(a),R_M(m))}&=&\mathrm{R_M(\mathbf{l}(R(a),m)+\mathbf{l}(a,R_M(m)))+\kappa \mu(a,m)},\label{rep mRBA1}\\
	\mathrm{\mathbf{r}(R_M(u),R(a))}&=&\mathrm{R_M(\mathbf{r}(R_M(m),a)+\mathbf{r}(m,R(a)))+\kappa r(m,a)}.\label{rep mRBA2}
\end{eqnarray}
If $(\mathcal{A},\mathrm{R})$ is a modified Rota-Baxter associative algebra of weight $\kappa$, then the couple $(\mathcal{A},\mathrm{R})$ is a bimodule over it, where the $\mathcal{A}$-bimodule structure on $\mathcal{A}$ is given by the algebra multiplication. This is called the adjoint bimodule.\\
Next we introduce the notion of \textbf{modified Rota-Baxter AssDer pairs of weight $\kappa$}.
\begin{defi}
	A modified Rota-Baxter AssDer pair of weight $\kappa$ consists of a modified Rota-Baxter associative algebra of weight $\kappa$ equipped with a derivation $\mathrm{d}$ on $\mathcal{A}$ such that
	\begin{equation}\label{equation RBAD1}
		\mathrm{R\circ d=d\circ R}.
	\end{equation}
	Denote it by $(\mathrm{A},\mu,\mathrm{R,d})$ or simply $(\mathcal{A},\mathrm{R,d})$.
\end{defi}
In the sequel, we use the notation modified Rota-Baxter AssDer pair instead of modified Rota-Baxter AssDer pair of weight $\kappa$ if there is no confusion.
\begin{pro}
	Let $(\mathrm{\mathcal{A}=(A,\mu_A),R_A,d_A})$ and $(\mathrm{\mathcal{B}=(B,\mu_B),R_B,d_B})$ be two modified Rota-Baxter AssDer pairs. Then $(\mathrm{A\oplus B,\mu_{A\oplus B},R_{A\oplus B},d_{A\oplus B}})$ is a modified Rota-Baxter AssDer pair where
	\begin{equation*}
		\mathrm{\mu_{A\oplus B}(a+x,b+y)=\mu(a,b)_A+\mu(x,y)_B},
	\end{equation*}
	\begin{equation*}
		\mathrm{R_{A\oplus B}(a+x)=R_A(a)+R_B(x)},
	\end{equation*}
	and
	\begin{equation*}
		\mathrm{d_{A\oplus B}(a+x)=d_A(a)+d_B(x),\quad \forall a,b\in A \text{ and } x,y\in B}.
	\end{equation*}
\end{pro}
\begin{proof}
	Let $\mathrm{a,b,c\in A}$ and $\mathrm{x,y,z\in B}$,
	\begin{align*}
		\mathrm{\mu_{A\oplus B}(\mu_{A\oplus B}(a+x,b+y),c+z)}&=\mathrm{\mu_{A\oplus B}(\mu_A(a,b)+\mu_B(x,y),c+z)}\\
		=&\mathrm{\mu_A(\mu_A(a,b),c)+\mu_B(\mu_B(x,y),z)}\\
		=&\mathrm{\mu_A(a,\mu_A(b,c))+\mu_B(x,\mu_B(y,z))}\\
		=&\mathrm{\mu_{A\oplus B}(a+x,\mu_A(b,c)\mu_B(y,z))}\\
		=&\mathrm{\mu_{A\oplus B}(a+x,\mu_{A\oplus B}(b+y,c+z)).}
	\end{align*}
	This means that $\mathrm{(A\oplus B,\mu_{A\oplus B})}$ is an associative algebra.\\
	\begin{align*}
		\mathrm{\mu_{A\oplus B}(R_{A\oplus B}(a+x),R_{A\oplus B}(b+y))}&=\mathrm{\mu_{A\oplus B}(R_A(a)+R_B(x),R_A(b)+R_B(y))}\\
		=&\mathrm{\mu_A(R_A(a),R_A(b))+\mu_B(R_B(x),R_B(y))}\\
		=&\mathrm{R_A(\mu_A(R_A(a),b))+R_B(\mu_B(R_B(x),y))+R_A(\mu_A(a,R_A(b)))+R_B(\mu_B(x,R_B(y)))}\\
		&+\mathrm{\kappa(\mu_A(a,b)+\mu_B(x,y))}\\
		=&\mathrm{R_{A\oplus B}(\mu_A(R_A(a),b)+\mu_B(R_B(x),y))+R_{A\oplus B}(\mu_A(a,R_A(b))+\mu_B(x,R_B(y)))}\\
		&+\mathrm{+\kappa(\mu_A(a,b)+\mu_B(x,y))}\\
		=&\mathrm{R_{A\oplus B}\Big(\mu_{A\oplus B}(R_{A\oplus B}(a+x),b+y)+\mu_{A\oplus B}( a+x,R_{A\oplus B}(b+y))\Big)}\\
		&+\mathrm{\kappa\mu_{A\oplus B}(a+x,b+y)}
	\end{align*}
	This means that $(A\oplus B,\mu_{A\oplus B},R_{A\oplus B})$ is a modified Rota-Baxter associative algebra.\\
	\begin{align*}
		\mathrm{d_{A\oplus B}(\mu_{A\oplus B}(a+x,b+y))}&=\mathrm{d_{A\oplus B}(\mu_A(a,b)+\mu_B(x,y))}\\
		&=\mathrm{d_A(\mu_A(a,b))+d_B(\mu_B(x,y))}\\
		&=\mathrm{\mu_A(d_A(a),b)+\mu_B(d_B(x),y)+\mu_A(a,d_A(b))+\mu_B(x,d_B(y))}\\
		&=\mathrm{\mu_{A\oplus B}(d_A(a)+d_B(x),b+y)+\mu_{A\oplus B}(a+x,d_A(b)+d_B(y))}\\
		&=\mathrm{\mu_{A\oplus B}(d_{A\oplus B}(a+x),b+y)+\mu_{A\oplus B}(a+x,d_{A\oplus B}(b+y))}.
	\end{align*}
	Which means that $\mathrm{d_{A\oplus B}}$ is a derivation on the associative algebra $\mathrm{(A\oplus B,\mu_{A\oplus B})}$.\\
	Finally
	\begin{align*}
		\mathrm{R_{A\oplus B}\circ d_{A\oplus B}(a+x)}&=\mathrm{R_{A\oplus B}(d_A(a)+d_B(x))}\\
		&=\mathrm{R_A(d_A(a))+R_B(d_B(x))}\\
		&=\mathrm{d_A(R_A(a))+d_B(R_B(x))}\\
		&=\mathrm{d_{A\oplus B}\circ R_{A\oplus B}(a+x)}.
	\end{align*}
This complete the proof.
\end{proof}
\begin{defi}
	Let $(\mathcal{A},\mathrm{R,d})$ be a modified Rota-Baxter AssDer pair. A \textbf{bimodule} of $(\mathcal{A},\mathrm{R,d})$ over a vector space $\mathcal{M}$ is a triple  $(\mathrm{\mathcal{M},R_M,d_M})$ where $\mathrm{(\mathcal{M},R_M)}$ is a bimodule of $(\mathcal{A},\mathrm{R})$ and $\mathrm{d_M:M\rightarrow M}$ is a linear map such that for all $\mathrm{a\in A,m\in M}$,
	\begin{eqnarray}
		\mathrm{d_M}\mathrm{(\mathbf{l}(a,m))}&=&\mathrm{\mathbf{l}(d(a),m)+\mathbf{l}(a,d_M(m))},\label{rep mRBAD1}\\
		\mathrm{d_M(\mathbf{r}(m,a))}&=&\mathrm{\mathbf{r}(d_M(m),a)+\mathbf{r}(m,d(a))},\label{rep mRBAD2}\\
		\mathrm{R_M\circ d_M}&=&\mathrm{d_M\circ R_M}.\label{rep mRBAD3}
	\end{eqnarray}
\end{defi}
\begin{ex}
	Let $(\mathcal{A},\mathrm{T,d})$ be a Rota-Baxter AssDer pair of weight $\kappa$ ( also denoted by $\kappa$-weighted Rota-Baxter AssDer pairs, see \cite{B1} for more details) and $\mathrm{(\mathcal{M},T_M,d_M)}$ be a bimodule over it.
	 Then $(\mathrm{\mathcal{M},R_M=\kappa Id_M+2T_M,d_M})$ is a bimodule over the modified Rota-Baxter AssDer pair \\ $(\mathrm{\mathcal{A},R=\kappa+Id_A+2T,d})$ of weight $-\kappa^2$.
\end{ex}
\begin{defi}
	Let $(\mathcal{A}_1=\mathrm{(A_1,\mu_1),R_1,d_1})$ and $(\mathcal{A}_2=\mathrm{(A_2,\mu_2),R_2,d_2})$ be two modified Rota-Baxter AssDer pairs. $\mathrm
	{f:(\mathcal{A}_1,R_1,d_1)\rightarrow (\mathcal{A}_2,R_2,d_2)}$ is said to be a \textbf{homomorphism} of modified Rota-Baxter AssDer pairs if $\mathrm{f:\mathcal{A}_1\rightarrow\mathcal{A}_2}$ is an associative algebra homomorphism such that
	\begin{eqnarray}
		\mathrm{f\circ R_1}&=&\mathrm{R_2\circ f},\\
		\mathrm{f\circ d_1}&=&\mathrm{d_2\circ f}.
	\end{eqnarray}
\end{defi}
Next, we define the \textbf{semi-direct product of modified Rota-Baxter AssDer pair} by a bimodule of it.
\begin{pro}
	Let $(\mathcal{A},\mathrm{R,d})$ be a modified Rota-Baxter AssDer pair and $\mathrm{(\mathcal{M},R_M,d_M)}$ be a bimodule of it. Then $\mathrm{(\mathcal{A}\oplus \mathcal{M},R_\ltimes,d_\ltimes)}$ is a modified Rota-Baxter AssDer pair, where the multiplication on $\mathrm{A\oplus M}$ is given by
	\begin{equation*}
		\mathrm{\mu_\ltimes(a+m,b+n)}=\mathrm{\mu(a,b)+\mathbf{l}(a,n)+\mathbf{r}(m,b)},
	\end{equation*}
	and the modified Rota-Baxter operator is given by
	\begin{equation*}
		\mathrm{R_\ltimes(a+m)=R(a)+R_M(m)},
	\end{equation*}
	and the derivation on $\mathrm{A\oplus M}$ is given by
	\begin{equation*}
		\mathrm{d_\ltimes(a+m)=d(a)+d_M(m)},\quad \forall \mathrm{a,b\in A} \text{ and } \mathrm{m,n\in M}.
	\end{equation*}
	We call such structure by the semi-direct product of the modified Rota-Baxter AssDer pair $\mathrm{(\mathcal{A},R,d)}$ by a bimodule of it $\mathrm{(\mathcal{M},R_M,d_M)}$ and denote it by $\mathcal{A}\ltimes_\mathrm{mRBAD}\mathcal{M}$.
\end{pro}
\begin{proof}
	According to \cite{D21} we have $\mathrm{(\mathcal{A}\oplus \mathcal{M},R_\ltimes)}$ is a modified Rota-Baxter associative algebra. So we need, first, to prove that $\mathrm{d_\ltimes}$ is a derivation of $\mathcal{A}\oplus \mathcal{M}$. Let $a\in A$ and $m\in M$, and by using equations \eqref{rep mRBAD1} and \eqref{rep mRBAD2} we have
	\begin{align*}
		\mathrm{d_\ltimes( \mu_\ltimes(a+m,b+n) )}&=\mathrm{d_\ltimes(\mu(a,b)+\mathbf{l}(a,n)+\mathbf{r}(m,b))}\\
		=&\mathrm{d(\mu(a,b))+d_M(\mathbf{l}(a,n))+d_M(\mathbf{r}(m,b))}\\
		=&\mathrm{\mu(d(a),b) +\mathbf{l}(d(a),n)+\mathbf{r}(d_M(m),b)}\\
		&+\mathrm{\mu(a,d(b)) +\mathbf{l}(a,d_M(n))+\mathbf{r}(m,d(b))}\\
		=& \mathrm{\mu_\ltimes(d_\ltimes(a+m),b+n) + \mu_\ltimes(a+m,d_\ltimes(b+n))}.
	\end{align*}
	And using \eqref{equation RBAD1} and \eqref{rep mRBAD3} we have
	\begin{align*}
		\mathrm{R_\ltimes\circ d_\ltimes(a+m)}&=\mathrm{R_\ltimes(d(a)+d_M(m))}\\
		=&\mathrm{R\circ d(a)+R_M\circ d_M(m)}\\
		=&\mathrm{d\circ R(a)+d_M\circ R_M(m)}\\
		=&\mathrm{d_\ltimes\circ R_\ltimes(a+m)}.
	\end{align*}
	This complete the proof.
\end{proof}
Recall that if $\mathrm{(A,\mu)}$ be an associative algebra then a Lie algebra structure rises via the commutator. we generalize this construction to the case of modified Rota-Baxter LieDer pair structures.
\begin{pro}\label{symmetrization}
	Let $\mathrm{(A,\mu,R,d)}$ be a modified Rota-Baxter AssDer pair. Define $\mathrm{[-,-]_C:A\otimes A\rightarrow A}$ by
	\begin{equation*}
		\mathrm{[a,b]_C=\mu(a,b)-\mu(b,a)},\quad \forall a,b\in A.
	\end{equation*}
	Then $\mathrm{(A,[-,-]_C,R,d)}$ is a modified Rota-Baxter LieDer pair, see \cite{B1} for more details. We denote it by $\mathrm{(\mathcal{A}_C,R,d)}$.
\end{pro}
\begin{proof}
	It is clear that $\mathrm{(A,[-,-]_C)}$ is a Lie algebra. Let $\mathrm{a,b\in A}$
	\begin{align*}
		\mathrm{[R(a),R(b)]_C}=&\mathrm{\mu(R(a),R(b))-\mu(R(b),R(a))}\\
		=&\mathrm{R(\mu(R(a),b)+\mu(a,R(b)))+\kappa\mu(a,b)}\\
		&\mathrm{-R(\mu(R(b),a)+\mu(b,R(a)))-\kappa\mu(b,a)}\\
		=&\mathrm{R(\mu(R(a),b)-\mu(b,R(a)))+R(\mu(a,R(b))-\mu(R(b),a))}\\
		&+\mathrm{\kappa(\mu(a,b)-\mu(b,a))}\\
		=&\mathrm{R([R(a),b]_C+[a,R(b)]_C)+\kappa [a,b]_C}
	\end{align*}
	This means that $\mathrm{(A,[-,-]_C,R)}$ is a modified Rota-Baxter Lie algebra. Also $\mathrm{d}$ is a derivation on $\mathrm{(A,[-,-]_C)}$ since $\mathrm{d}$ is a derivation on $(\mathrm{A,\mu})$, additionally the equation \eqref{equation RBAD1} holds. \\
	This complete the proof.
\end{proof}
In the following, we show that a modified Rota-Baxter AssDer pair induces a new AssDer pair structure, which called \textbf{the induced modified Rota-Baxter AssDer pair}.
\begin{pro}
	Let $\mathrm{(A,\mu,R,d)}$ be a modified Rota-Baxter AssDer pair. Then we have
	\begin{enumerate}
		\item[(1)] $(\mathrm{A,\mu_R,d})$ is a new AssDer pair with
		\begin{equation}\label{induced multiplication}
			\mathrm{\mu_R(a,b)=\mu(R(a),b)+\mu(a,R(b))},\quad \forall \mathrm{a,b\in A}.
		\end{equation}
		We denote it by $\mathrm{(\mathcal{A}_R,d)}$.
		\item [(2)] The triple $(\mathrm{\mathcal{A}_R,R,d})$ is a modified Rota-Baxter AssDer pair.
	\end{enumerate}
\end{pro}
\begin{proof}
	For the first assertion, we already have $(\mathrm{A,\mu_R})$ is an associative algebra, so we need just to prove that $\mathrm{d}$ is a derivation on it. Using equation \eqref{equation RBAD1} and the fact that $\mathrm{d}$ is a derivation on $\mathrm{(A,\mu)}$ we obtain the result.\\
	For the second assertion, let $\mathrm{a,b\in A}$,
	\begin{align*}
		\mathrm{\mu_R(R(a),R(b))}=&\mathrm{\mu(R^2(a),R(b)+\mu(R(a),R^2(b)))}\\
		=&\mathrm{R\big(\mu(R^2(a),b)+\mu(R(a),R(b))\big)+\kappa\mu(R(a),b)}\\
		&+\mathrm{R\big(\mu(R(a),R(b)+\mu(a,R^2(b)))\big)+\kappa\mu(a,R(b))}\\
		=&\mathrm{R(\mu_R(R(a),b))+\kappa\mu(R(a),b)+R(\mu_R(a,R(b)))+\kappa\mu(a,R(b))}\\
		=&\mathrm{R(\mu_R(R(a),b)+\mu_R(a,R(b)))+\kappa \mu_R(a,b)}
	\end{align*}
	This means that $\mathrm{R}$ is a modified Rota-Baxter operator on $\mathrm{(A,\mu_R)}$ and since the equation \eqref{equation RBAD1} holds we complete the proof.
\end{proof}
\begin{thm}
	Let $(\mathrm{\mathcal{A},R,d})$ be a modified Rota-Baxter AssDer pair and $\mathrm{(M,\mathbf{l,r},R_M,d_M)}$ be a bimodule of it. Define linear maps for $\mathrm{a\in A}$ and $\mathrm{m\in M}$,
	\begin{eqnarray}
		\mathrm{\widetilde{\mathbf{l}}(a,m)}&=&\mathrm{\mathbf{l}(R(a),m)-R_M\circ \mathbf{l}(a,m)},\label{induced rep1}\\
		\mathrm{\widetilde{\mathbf{r}}(m,a)}&=&\mathrm{\mathbf{r}(m,R(a))-R_M\circ \mathbf{r}(m,a)}.\label{induced rep2}
	\end{eqnarray}
	Then $(\mathrm{M,\mathbf{\widetilde{l},\widetilde{r}},d_M})$ define a bimodule of the AssDer pair $\mathrm{(\mathcal{A}_R,d)}$. Moreover, $\mathrm{(M,\mathbf{\widetilde{l},\widetilde{r}},R_M,d_M)}$ is a bimodule of the modified Rota-Baxter AssDer pair $\mathrm{(\mathcal{A}_R,R,d)}$, we denote it by $\mathrm{(\widetilde{\mathcal{M}},R_M,d_M)}$ where $\mathrm{\widetilde{\mathcal{M}}=(M,\mathbf{\widetilde{l},\widetilde{r}})}$.
\end{thm}
\begin{proof}
	According to (\cite{D21}, proposition 2.12) $\mathrm{(M,\mathbf{\widetilde{l},\widetilde{r}},R_M)}$
	is a bimodule of $\mathrm{(\mathcal{A}_R,R)}$. Let $\mathrm{a\in A,m\in M}$ and using equations \eqref{equation RBAD1} and \eqref{rep mRBAD3} we have
	\begin{align*}
		\mathrm{d_M\circ \widetilde{\mathbf{l}}(a,m)}&=\mathrm{d_M(\mathbf{l}(R(a),m)-R_M\circ \mathbf{l}(a,m))}\\
		&=\mathrm{d_M\circ \mathbf{l}(R(a),m)-d_M\circ R_M\circ \mathbf{l}(a,m)}\\
		&=\mathrm{\mathbf{l}(d\circ R(a),m)+\mathbf{l}(R(a),d_M(m))-R_M\circ \mathbf{l}(d(a),m)-\mathbf{l}(a,d_M(m))}\\
		&=\mathrm{\mathbf{l}(R\circ d(a),m)-R_M\circ \mathbf{l}(d(a),m)+\mathbf{l}(R(a),d_M(m))-R_M\circ \mathbf{l}(a,d_M(m))}\\
		&=\mathrm{\widetilde{\mathbf{l}}(d(a),m)+\widetilde{\mathbf{l}}(a,d_M(m))}.
	\end{align*}
	Similarly to $\widetilde{\mathbf{r}}$. This complete the proof.
\end{proof}
\section{Cohomologies of modified Rota-Baxter AssDer pairs}
\def\theequation{\arabic{section}.\arabic{equation}}
\setcounter{equation} {0}
In this section, we study cohomologies of modified Rota-Baxter AssDer pairs in an arbitrary bimodule. To enhance clarity in our construction, we recall cohomologies of modified Rota-Baxter associative algebras with coefficients in arbitrary bimodules..\\
Let $(\mathrm{\mathcal{M},R_M})$ be a bimodule of modified Rota-Baxter associative algebra $\mathrm{(\mathcal{A},R)}$. By (proposition 2.12 and propisition 2.9 \cite{Das2}), $\mathrm{(M,\mathbf{\widetilde{l},\widetilde{r}},R_M)}$ is a bimodule over the induced modified Rota-Baxter associative algebra $(\mathrm{\mathcal{A}_R,R})$ where $\mathrm{\mu_R,\mathbf{\widetilde{l},\widetilde{r}}}$ are given respectively by equations \eqref{induced multiplication}, \eqref{induced rep1} and \eqref{induced rep2}.

Define the $n$-cochains group $\mathrm{C^n_{mRBAA^\kappa}(A;M)}$ as follows
\begin{center}
	For $\mathrm{n=1}$, \quad $\mathrm{C^1_{mRBAA^\kappa}(A;M)=Hom(A,M)}$ and
\end{center}
\begin{equation*}
	\mathrm{C^n_{mRBAA^\kappa}(A;M)=C^n(A;M)\oplus C^{n-1}(A;M)},\quad \mathrm{n\geq2}.
\end{equation*}
The coboundary operator
\begin{equation*}
	\mathrm{\partial^n_{mRBAA^\kappa}:C^n_{mRBAA^\kappa}(A;M)\rightarrow C^{n+1}_{mRBAA^\kappa}(A;M)} \text{ is given by }
\end{equation*}
\begin{equation}
	\mathrm{\partial^n_{mRBAA^\kappa}}(\mathrm{f_n,g_{n-1}})=\Big(\mathrm{\delta^n_{Hoch}}(\mathrm{f_n}),-\mathrm{\delta^{n-1}_{mHoch}}\mathrm{(g_{n-1})-\mathrm{\phi^n}(f_n)}\Big),
\end{equation}
where
\begin{align*}
	\mathrm{\delta^n_{Hoch}}&:\mathrm{C^n(A;M)\rightarrow C^{n+1}(A;M)},\\
	\mathrm{\delta^{n-1}_{mHoch}}&:\mathrm{C^{n-1}(A_R;M)\rightarrow C^n(A_R;M)},\\
	\mathrm{\phi^n}&:\mathrm{C^n(A;M)\rightarrow C^{n}(A_R;M)},
\end{align*}
are defined respectively by:\\
\begin{align*}
	&\mathrm{\delta^n_{Hoch}(f_n)}\mathrm{(a_1,\cdots,a_{n+1})}\\
	=&\mathrm{(-1)^{n+1}\mathbf{l}(a_1,}\mathrm{f_n}(\mathrm{a_2,\cdots,a_{n+1}}))+\mathrm{\mathbf{r}}(\mathrm{f_n}(\mathrm{a_1,\cdots,a_n}),\mathrm{a_{n+1}})\\
	&+\mathrm{\displaystyle\sum_\mathrm{{i=1}}^\mathrm{n}\mathrm{(-1)^{i+n+1}}f_\mathrm{n}(\mathrm{a_1,\cdots,\mu(a_i,a_{i+1}),\cdots,a_{n+1}})},\\
	&\mathrm{\delta^n_{mHoch}}(\mathrm{f_n})\mathrm{(a_1,\cdots,a_{n+1})}\\
	=&\mathrm{(-1)^{n+1}\mathbf{l}(R(a_1),}\mathrm{f_n}(\mathrm{a_2,\cdots,a_{n+1}}))-\mathrm{(-1)^{n+1}R_M(\mathbf{l}(a_1,}\mathrm{f_n}\mathrm{(a_2,\cdots,a_{n+1})}))\\
	&+\mathrm{\mathbf{r}}(\mathrm{f_n}(\mathrm{a_1,\cdots,a_n}),\mathrm{R_M(a_{n+1})})-\mathrm{R_M}(\mathrm{\mathbf{r}}(\mathrm{f_n}(\mathrm{a_1,\cdots,a_n}),\mathrm{a_{n+1}}))\\
	&+\displaystyle\sum_\mathrm{{i=1}}^\mathrm{n}\mathrm{(-1)^{i+n+1}}\mathrm{f_n}(\mathrm{a_1,\cdots,\mu(R(a_i),a_{i+1})+\mu(a_i,R(a_{i+1})),\cdots,a_{n+1}}),\\
	&\mathrm{\phi^n(f_n)}(\mathrm{a_1,a_2,\ldots,a_n})\\
	=&\mathrm{f_n}(\mathrm{Ra_1,Ra_2,\ldots, Ra_n})\\
	&-\sum _{1\leq i_1<i_2< \cdots < i_r\leq n ,r~ \mbox{odd}}{(-\kappa)^{\frac{r-1}{2}}} (\mathrm{R_V}\circ \mathrm{f_n})(\mathrm{R(a_1),\ldots,a_{i_1},\ldots,a_{i_r},\ldots,R(a_n}))\\
	&-\sum _{1\leq i_1<i_2< \cdots < i_r\leq n ,r~ \mbox{even}}{(-\kappa)^{\frac{r}{2}+1}} (\mathrm{R_V}\circ \mathrm{f_n})(\mathrm{R(a_1),\ldots,a_{i_1},\ldots,a_{i_r},\ldots,R(a_n))}.
\end{align*}
Where $(\mathrm{C^\star(A;M),\delta^\star_{Hoch}})$ defines the Hochshild cochain complex and we denote its corresponding cohomology groups by $\mathrm{\mathcal{H}^\star(A;M)}$.\\
And $(\mathrm{C^\star(A_R;M),\delta^\star_{mHoch}})$ defines the Hochshild cochain complex  of the induced associative algebra $\mathrm{\mathcal{A}_R}$ with coefficients in the bimodule $\mathrm{\widetilde{\mathcal{M}}}$. We denote its cohomology groups by $\mathrm{\mathcal{H}^\star(A_R;M)}$.\\
Finaly $\phi^\star$ defines a morphism of cochain complexes from $(\mathrm{C^\star(A;M),\delta^\star_{Hoch}})$ to $(\mathrm{C^\star(A_R;M),\delta^\star_{mHoch}})$.\\

According to the cochain complex $(\mathrm{C^\star_{mRBAA^\kappa}(A;M),\partial^\star_{mRBAA^\kappa}})$ the set of all $n$-cocycles and $n$-coboundaries are denoted by, respectively, $\mathrm{\mathcal{Z}^n_{mRBAA^\kappa}(A;M)}$ and $\mathrm{\mathcal{B}^n_{mRBAA^\kappa}(A;M)}$. Define the corresponding cohomology group by
\begin{equation*}
	\mathrm{\mathcal{H}^n_{\mathrm{mRBAA^{\kappa}}}(A,M):=\frac{\mathcal{Z}^n_{\mathrm{mRBAA^{\kappa}}}(A,M)}
	{\mathcal{B}^n_{\mathrm{mRBAA^{\kappa}}}(A,M)},\quad \text{for } n\geq1},
\end{equation*}
which is called the $n$-cohomology group of $(\mathrm{\mathcal{A},R})$ with coefficients in the bimodule $(\mathrm{\mathcal{M},R_M})$.\\
Define a linear map

\begin{equation*}
	\mathrm{\Delta^n:C^n_\mathrm{mRBAA^\kappa}(A;M)\rightarrow C^n_\mathrm{mRBAA^\kappa}(A;M)\quad \text{ by }}
\end{equation*}
\begin{equation}\label{coboundary derivation}
	\mathrm{\Delta^n(f_n,g_{n-1}):=(\Delta^n(f_n),\Delta^{n-1}(g_{n-1})),\quad \forall (f_n,g_{n-1})}\in C^n_\mathrm{mRBAA^\kappa}(\mathrm{A;M}).
\end{equation}
Where
\begin{equation*}
	\mathrm{\Delta^n (f_n)=\displaystyle\sum_{i=1}^nf_n\circ (\mathrm{Id_A}\otimes \cdots\otimes d\otimes \cdots\otimes \mathrm{Id_A})-d_{\mathrm{M}}\circ f_n}.
\end{equation*}
\begin{pro}
	With the previous result, we have the following
	\begin{equation}\label{coboundary3}
		\mathrm{\phi^n\circ\Delta^n=\Delta^n\circ\phi^n,\quad \forall n\geq1.}
	\end{equation}
\end{pro}
\begin{proof}
	For $f_n\in C^n(A;M)$ and $\mathrm{a_1,\cdots,a_n}\in A$ we have
	\begin{align*}
		&\mathrm{\phi^n\circ\Delta^n(f_n)(\mathrm{a_1,\cdots,a_n})}\\
		=&\mathrm{\Delta^n(f_n(\mathrm{R(a_1)},\cdots\mathrm{R(a_n)}))}\\
		&\mathrm{-\sum _{1\leq i_1<i_2< \cdots < i_r\leq n ,r~ \mbox{odd}}{(-\kappa)^{\frac{r-1}{2}}} (R_M\circ\Delta^n f_n)(\mathrm{R(a_1),\ldots,a_{i_1},\ldots,a_{i_r},\ldots,R(a_n)})}\\
		&\mathrm{-\sum _{1\leq i_1<i_2< \cdots < i_r\leq n ,r~ \mbox{even}}{(-\kappa)^{\frac{r}{2}+1}} (R_M\circ\Delta^n f_n)(\mathrm{R(a_1),\ldots,a_{i_1},\ldots,a_{i_r},\ldots,R(a_n)})}\\
		=&\mathrm{\displaystyle\sum_{\mathrm{k=1}}^nf_n(\mathrm{R(a_1),\cdots,d\circ R(a_i),\cdots,R(a_n)})-\mathrm{d_M}\circ f_n\mathrm{(R(a_1),\cdots,R(a_n))}}\\
		&-\mathrm{\sum _{1\leq i_1<i_2< \cdots < i_r\leq n ,r~ \mbox{odd}}{(-\kappa)^{\frac{r-1}{2}}}\Big( \displaystyle\sum_{\mathrm{k\neq i_1,\cdots,i_r}}(\mathrm{R_M}\circ f_n)\mathrm{(R(a_1),\cdots,d\circ R(a_k),\cdots,a_{i_1},\cdots,a_{i_r},\cdots,R(a_n))}}\\
		&+\mathrm{(\mathrm{R_M}\circ f_n)\displaystyle\sum_{\mathrm{p=1}}^r(\mathrm{R(a_1),\cdots,a_{i_1},\cdots,d_M(a)_{i_p},\cdots,a_{i_r},\cdots,R(a_n)})}\\
		&-\mathrm{(\mathrm{R_M}\circ \mathrm{d_M}f_n)(\mathrm{(R(a_1),\cdots,a_{i_1},\cdots,a_{i_r},\cdots,R(a_n))})\Big)}\\
		&-\mathrm{\sum _{1\leq i_1<i_2< \cdots < i_r\leq n ,r~ \mbox{even}}{(-\kappa)^{\frac{r}{2}+1}}\Big( \displaystyle\sum_{\mathrm{k\neq i_1,\cdots,i_r}}(\mathrm{R_M}\circ f_n)\mathrm{(R(a_1),\cdots,d\circ R(a_k),\cdots,a_{i_1},\cdots,a_{i_r},\cdots,R(a_n))}}\\
		&+\mathrm{(\mathrm{R_M}\circ f_n)\displaystyle\sum_{\mathrm{p=1}}^r(\mathrm{R(a_1),\cdots,a_{i_1},\cdots,d(a_{i_p}),\cdots,a_{i_r},\cdots,Ra_n})}\\
		&-\mathrm{(\mathrm{R_M}\circ \mathrm{d_M}f_n)(\mathrm{(R(a_1),\cdots,a_{i_1},\cdots,a_{i_r},\cdots,R(a_n))})\Big)}\\
		=&\mathrm{\Delta^n\circ\phi^n(f_n)(\mathrm{a_1,\cdots,a_n}})
	\end{align*}
\end{proof}
	Recall that, from the cohomology of AssDer pair (\cite{A2}. Lemma 1), we have
\begin{equation}\label{coboundary1}
	\mathrm{\delta^n_\mathrm{Hoch}\circ\Delta^n=\Delta^{n+1}\circ\delta^n_\mathrm{Hoch}}.
\end{equation}

Also since $\mathrm{\mathcal{A}_{\mathrm{R}}}$ is an associative algebra and
$\mathrm{\delta_\mathrm{mHoch}^\star}$ its coboundary with respect to the bimodule
$\mathrm{(M,\mathbf{\widetilde{l},\widetilde{r}})}$ and  $\mathrm{(\mathcal{A}_{\mathrm{R}},d)}$ is an AssDer pair then we get
\begin{equation}\label{coboundary2}
	\mathrm{\delta^n_\mathrm{mHoch}\circ\Delta^n=	\Delta^{n+1}\circ\delta^n_\mathrm{mHoch}}.
\end{equation}
	\begin{pro}
	With the above notations, $\Delta^\star$ is a cochain map, i.e.
	\begin{equation}\label{coboundary5}
		\mathrm{\partial^n_\mathrm{mRBAA^\kappa}\circ \Delta^n
			=\Delta^{n+1}\circ\partial^n_\mathrm{mRBAA^\kappa}},\quad \forall n\geq1.
	\end{equation}
\end{pro}
\begin{proof}
	For $\mathrm{(f_n,g_{n-1})\in C^n_\mathrm{mRBAA^\kappa}(A;M)}$ and by using \eqref{coboundary1}, \eqref{coboundary2}, \eqref{coboundary3} and \eqref{coboundary derivation} we have
	\begin{align*}
		\mathrm{\partial^n_\mathrm{mRBAA^\kappa}(\Delta^n(f_n,g_{n-1}))}
		&=\mathrm{\partial^n_\mathrm{mRBAA^\kappa}(\Delta^nf_n,\Delta^{n-1}g_{n-1})}\\
		&=\mathrm{\Big(\delta^n_\mathrm{Hoch}(\Delta^nf_n),-\delta^{n-1}_\mathrm{mHoch}
			(\Delta^{n-1}g_{n-1})-\phi^n(\Delta^nf_n)\Big)}\\
		&\mathrm{=\Big(\Delta^{n+1}(\delta^n_\mathrm
			{Hoch}f_n),-\Delta^n(\delta^{n-1}_\mathrm{mHoch}g_{n-1})-\Delta^n(\phi^nf_n)\Big)}\\
		&\mathrm{=\Delta^{n+1}(\partial^n_\mathrm{mRBAA^\kappa}(f_n,g_{n-1}))}.
	\end{align*}
This complete the proof.
\end{proof}
	Using all those tools we are in position to define the cohomology of modified Rota-Baxter AssDer pair
$\mathrm{(\mathcal{A},R,d)}$ with coefficients in a bimodule $\mathrm{(\mathcal{M},R_\mathrm{M},d_\mathrm{M})}$.\\
Denote
\begin{equation*}
	\mathrm{\mathfrak{C}^n_{\mathrm{mRBAD^{\kappa}}}(A;M):=C^n_\mathrm{mRBAA^\kappa}(A;M)\times
		C^{n-1}_{\mathrm{mRBAA^{\kappa}}}(A;M),\quad n\geq2},
\end{equation*}
and
\begin{equation*}
	\mathrm{\mathfrak{C}^1_{\mathrm{mRBAD^{\kappa}}}(A;M):=C^1(A;M)}.
\end{equation*}
Define a linear map
\begin{eqnarray*}
	\mathrm{\mathfrak{D}^1_{\mathrm{mRBAD^{\kappa}}}}&:&\mathrm{\mathfrak{C}^1_{\mathrm{mRBAD^{\kappa}}}
		(A;M)\rightarrow \mathfrak{C}^2_{\mathrm{mRBAD^{\kappa}}}(A;M) \text{ by}}\\
	&&\mathrm{\mathfrak{D}^1_{\mathrm{mRBAD^{\kappa}}}(f_1)
		=(\partial^1_\mathrm{mRBAA^\kappa}(f_1),-\Delta^1(f_1)),\quad \forall f_1\in
		C^1(A;M)},
\end{eqnarray*}
and when $\mathrm{n\geq2}$

\begin{equation*}
	\mathrm{\mathfrak{D}^n_{\mathrm{mRBAD^{\kappa}}}:\mathfrak{C}^n_{\mathrm{mRBAD^{\kappa}}}(A;M)\rightarrow
		\mathfrak{C}^{n+1}_{\mathrm{mRBAD^{\kappa}}}(A;M)}
\end{equation*}
is defined by
\begin{equation}\label{mRBLD coboundary}
	\mathrm{\mathfrak{D}^n_{\mathrm{mRBAD^{\kappa}}}((f_n,g_{n-1}),(h_{n-1},s_{n-2}))
		=(\partial^n_\mathrm{mRBAA^\kappa}(f_n,g_{n-1}),\partial^{n-1}_\mathrm{mRBAA^\kappa}(h_{n-1},s_{n-2})
		+(-1)^n\Delta^n(f_n,g_{n-1}))}.
\end{equation}
	\begin{thm}
	According to The above notations we have
	$\mathrm{\big(\mathfrak{C}^\star_\mathrm{mRBAD^\kappa}(A;M),\mathfrak{D}^\star_{\mathrm{mRBAD^{\kappa}}}\big)}$ is a cochain complex, i.e,
	\begin{equation*}
		\mathrm{\mathfrak{D}^{n+1}_{\mathrm{mRBAD^{\kappa}}}\circ \mathfrak{D}^n_{\mathrm{mRBAD^{\kappa}}}=0,
			\quad \forall n\geq1}.
	\end{equation*}
\end{thm}
\begin{proof}
	For $\mathrm{n\geq1}$, using equations \eqref{mRBLD coboundary},  \eqref{coboundary5}
	
	\begin{align*}
		&\mathrm{\mathfrak{D}^{n+1}_{\mathrm{mRBAD^{\kappa}}}\circ\mathfrak{D}^n_{\mathrm{mRBAD^{\kappa}}}
			((f_n,g_{n-1}),(h_{n-1},s_{n-2}))}\\
		=&\mathrm{\mathfrak{D}^{n+1}_{\mathrm{mRBAD^{\kappa}}}(\partial^n_\mathrm{mRBAA^\kappa}(f_n,g_{n-1}),
			\partial^{n-1}_\mathrm{mRBAA^\kappa}(h_{n-1},s_{n-2})+(-1)^n\Delta^n(f_n,g_{n-1}))}\\
		=&\mathrm{\Big(\partial^{n+1}_\mathrm{mRBAA^\kappa}(\partial^n_\mathrm{mRBAA^\kappa}(f_n,g_{n-1})),
			\partial^n_\mathrm{mRBAA^\kappa}(\partial^{n-1}_\mathrm{mRBAA^\kappa}(h_{n-1},s_{n-2})
			+(-1)^n\Delta^n(f_n,g_{n-1})}\\
		&\mathrm{+(-1)^{n+1}\Delta^{n+1}(\partial^n_\mathrm{mRBAA^\kappa}(f_n,g_{n-1})\Big)}\\
		=&\mathrm{\Big((0,0),(-1)^n\partial^n_\mathrm{mRBAA^\kappa}(\Delta^n(f_n,g_{n-1}))
			+(-1)^{n+1}\Delta^{n+1}(\partial^n_\mathrm{mRBAA^\kappa}(f_n,g_{n-1})\Big)}\\
		&=0.
	\end{align*}
	This complete the proof.
\end{proof}
	With respect to the bimodule $\mathrm{(\mathcal{M},\mathrm{R_M},\mathrm{d_M})}$ we obtain a complex
$\mathrm{\Big(\mathfrak{C}^{\star}_{\mathrm{mRBAD^{\kappa}}}
	(A,M),\mathfrak{D}^\star_{\mathrm{mRBAD^{\kappa}}}\Big)}$. Let $\mathrm{Z^n_{\mathrm{mRBAD^{\kappa}}}(A,M)}$ and
$\mathrm{B^n_{\mathrm{mRBAD^{\kappa}}}(A,M)}$ denote the space of $\mathrm{n}$-cocycles and $\mathrm{n}$-coboundaries, respectively.
Then we define the corresponding cohomology groups by
\begin{equation*}
	\mathrm{\mathcal{H}^n_{\mathrm{mRBAD^{\kappa}}}(A,M):=\frac{Z^n_{\mathrm{mRBAD^{\kappa}}}(A,M)}
	{B^n_{\mathrm{mRBAD^{\kappa}}}(A,M)},\quad \text{for } n\geq1.}
\end{equation*}
They are called \textbf{the cohomology of modified Rota-Baxter AssDer pair} $\mathrm{(\mathcal{A},R,d)}$ with coefficients
in the bimodule $\mathrm{(\mathcal{M},R_M,d_{\mathrm{M}})}$.
\newline

\noindent{\textbf{Relation with the cohomology of modified Rota-Baxter LieDer pairs.}} Now we study \textbf{connection} between the cohomology of modified Rota-Baxter AssDer pair and the cohomology of modified Rota-Baxter LieDer pair. By proposition \eqref{symmetrization}, we have $(\mathrm{\mathcal{A}_C,R,d})$ is a modified Rota-Baxter LieDer pair. In the following proposition we construct a representation of it.
\begin{pro}
	Let $(\mathrm{\mathcal{A},R,d})$ be a modified Rota-Baxter AssDer pair and $(\mathrm{\mathcal{M},R_M,d_M})$ be a module of it. Then $(\mathrm{M,\rho,R_M,d_M})$ is a representation of the modified Rota-Baxter LieDer pair $(\mathrm{\mathcal{A}_C,R,d})$ where the linear map $\rho:\mathrm{A\rightarrow gl(M)}$ is given by
	\begin{equation*}
		\mathrm{\rho(a)(m):=\mathbf{l}(a,m)-\mathbf{r}(m,a),\quad \forall a\in A,m\in M}.
	\end{equation*}
\end{pro}
\begin{proof}
	Let $\mathrm{a,b\in A}$ and $\mathrm{m\in M}$.
	\begin{align*}
		\mathrm{\rho(a)\rho(b)(m)-\rho(b)\rho(a)(m)}=&\mathrm{\rho(a)(\mathbf{l}(b,m)-\mathbf{r}(m,b))-\rho(b)(\mathbf{l}(a,m)-\mathbf{r}(m,a))}\\
		=&\mathrm{\mathbf{l}(a,\mathbf{l}(b,m))-\mathbf{r}(\mathbf{l}(b,m),a)-\mathbf{l}(a,\mathbf{r}(m,b))+\mathbf{r}(\mathbf{r}(m,b),a)}\\
		&\mathrm{-\mathbf{l}(b,\mathbf{l}(a,m))+\mathbf{r}(\mathbf{l}(a,m),b)+\mathbf{l}(b,\mathbf{r}(m,a))-\mathbf{r}(\mathbf{r}(m,a),b)}\\
		=&\mathrm{\rho([a,b]_C)(m)-\mathbf{l}(b,\mathbf{r}(m,a))-\mathbf{l}(a,\mathbf{r}(m,b))+\mathbf{l}(a,\mathbf{r}(m,b))+\mathbf{l}(b,\mathbf{r}(m,a))}\\
		=&\mathrm{\rho([a,b]_C)(m)}.
	\end{align*}
And
\begin{align*}
	\mathrm{d_M(\rho(a)(m))}&=\mathrm{d_M(\mathbf{l}(a,m)-\mathbf{r}(m,a))}\\
	&=\mathrm{\mathbf{l}(d(a),m)+\mathbf{l}(a,d_M(m))-\mathbf{r}(d_M(m),a)-\mathbf{r}(m,d(a))}\\
	&=\mathrm{\rho(d(a))(m)+\rho(a)d_M(m)}.
\end{align*}
And
\begin{align*}
	\mathrm{\rho(R(a),R_M(m))}=&\mathrm{\mathbf{l}(R(a),R_M(m))-\mathbf{r}(R_M(m),R(a))}\\
	=&\mathrm{R_M(\mathbf{l}(R(a),m)+\mathbf{l}(a,R_M(m)))+\kappa \mathbf{l}(a,m)}\\
	&-\mathrm{R_M(\mathbf{r}(R_M(m),a)+\mathbf{r}(m,R(a)))-\kappa \mathbf{r}(m,a)}\\
	=&\mathrm{R_M(\mathbf{l}(a,R_M(m))-\mathbf{r}(R_M(m),a)+\mathbf{l}(R(a),m)-\mathbf{r}(m,R(a)))}\\
	&+\mathrm{\kappa(\mathbf{l}(a,m)-\mathbf{r}(m,a))}\\
	=&\mathrm{R_M(\rho(a)(R_M(m))+\rho(R(a))(m))+\kappa \rho(a)(m)}.
\end{align*}
And since the equation \eqref{rep mRBAD3} holds because  $(\mathrm{\mathcal{M},R_M,d_M})$ is a bimodule over the modified Rota-Baxter AssDer pair $\mathrm{(\mathcal{A},R,d)}$ then we complete the proof.
\end{proof}
\begin{rmk}
	If $(\mathrm{\mathcal{A},R,d})$ is a modified Rota-Baxter AssDer pair and $(\mathrm{\mathcal{M},R_M,d_M})$ is a bimodule on it. Then $\mathrm{(\mathcal{A}_C)_R=(\mathcal{A}_R)_C}$ as a Lie algebra. Moreover, the representation of $\mathrm{(\mathcal{A}_C)_R}$ on $\mathrm{\widehat{\mathcal{M}}}$ and the representation of $\mathrm{(\mathcal{A}_R)_C}$ on $\mathrm{\widetilde{\mathcal{M}}}$ coincide.
\end{rmk}
There is a morphism from the Hochschild cochain
complex of an associative algebra $\mathrm{(A,\mu)}$ to the Chevalley-Eilenberg cochain complex of the corresponding skewsymmetrized Lie algebra $(\mathrm{A,[-,-]_C})$. Then the following diagrams commute
\begin{center}
$\xymatrix{
    \mathrm{C^n(A;M)}\ar[r]^{\mathrm{\delta_{Hoch}^n}} \ar[d]^{\mathrm{S_n}}& \mathrm{C^{n+1}(A;M)} \ar@{}[r] \ar[d]^{\mathrm{S_{n+1}}}
    &\mathrm{C^{n}(A_R;\widetilde{M})}\ar[r]^{\mathrm{\delta^n_\mathrm{mHoch}}} \ar[d]^{\mathrm{S_{n}}}&\mathrm{C^{n+1}(A_R;\widetilde{M})}\ar[d]^{\mathrm{S_{n+1}}} \ar@{}[r]& \\
    \mathrm{C^n(A_C;M)}\ar[r]^{\mathrm{\delta_{CE}^n}} & \mathrm{C^{n+1}(A_C;M)} \ar@{}[r]&\mathrm{C^{n}
    ((A_R)_C;\widetilde{M})}\ar[r]^{\mathrm{\delta^n_\mathrm{mRBO^\kappa}}}&\mathrm{C^{n+1}((A_R)_C;\widetilde{M})} \ar@{}[r]&
  }$
\end{center}
Here $\mathrm{\mathrm{S_\star}}$ are skew-symmetrization maps. Then we have the following
\begin{thm}
	Let $(\mathrm{\mathcal{M},R_M})$ be a bimodule of the modified Rota-Baxter associative algebra $(\mathrm{\mathcal{A}},R)$. Then the collection of maps
	$\mathrm{\mathcal{S}_n:C^n_{mRBAA^\kappa}(A;M)\rightarrow  C^n_{mRBLA^\kappa}(A_C;M)}$, $\mathrm{n\geq1}$ defined by
	\begin{equation*}
		\mathrm{\mathcal{S}_n=(S_n,S_{n-1})}
	\end{equation*}
gives rise to a morphism from the cohomology of $\mathrm{(\mathcal{A},R)}$ with coefficients in $\mathrm{(\mathcal{M},R_M)}$ to the cohomology of $\mathrm{(\mathcal{A}_C,R)}$ with coefficients in the representation $\mathrm{(M,\rho,R_M)}$.
\end{thm}
\begin{proof}
	We only need to prove that the set of maps $\mathrm{\{\mathcal{S}_n\}_{n\geq1}}$ commute with corresponding coboundary maps. Let $\mathrm{(f_n,f_{n-1})\in C^n_{mRBAA^\kappa}(A;M)}$,
	\begin{align*}
		\mathrm{(\partial^n_{mRBLA^\kappa}\circ\mathcal{S}_n)(f_n,f_{n-1})}&=\mathrm{\partial^n_{mRBAA^\kappa}(S_n(f_n),S_{n-1}(f_{n-1}))}\\
		&=\mathrm{(\delta^n_{CE}\circ S_n(f_n),-\delta_{mRBO^\kappa}^{n-1}\circ S_{n-1}(f_{n-1})-\phi^n\circ S_n(f_n))}\\
		&=(\mathrm{S_{n+1}\circ \delta_{Hoch}^n(f_n),-S_n\circ \delta_{mHoch}^{n-1}(f_{n-1})-S_n\circ \phi^n(f_n)})\\
		&=\mathrm{\mathcal{S}_n(\delta_{Hoch}^n(f_n),-\delta_{mHoch}^{n-1}(f_{n-1})-\phi^n(f_n))}\\
		&=\mathrm{(\mathcal{S}_{n+1}\circ \partial^n_{mRBAA^\kappa})(f_n,f_{n-1})}
	\end{align*}
\end{proof}
\begin{thm}
	Let $\mathrm{(\mathcal{A},R,d)}$ be a modified Rota-Baxter AssDer pair and $\mathrm{(\mathcal{M},R_M,d_M)}$ be a bimodule of it. Then the collection of maps $\mathrm{\mathfrak{S}_n:\mathfrak{C}^n_{mRBAD^\kappa}(A;M)\rightarrow  \mathfrak{C}^n_{mRBLD^\kappa}(A_C;M)}$ $\mathrm{n\geq1}$ defined by
	\begin{equation*}
		\mathrm{\mathfrak{S}_n:=(\mathcal{S}_n,\mathcal{S}_{n-1})}
	\end{equation*}
induces a morphism from the cohomology of $\mathrm{(\mathcal{A},R,d)}$ with coefficients in $\mathrm{(\mathcal{M},R_M,d_M)}$ to the cohomology of $\mathrm{(\mathcal{A}_C,R,d)}$ with coefficients in the representation $\mathrm{(M,\rho,R_M,d_M)}$.
\end{thm}
\begin{proof}
	Let $\mathrm{(F_n,G_{n-1})\in C^n_{mRBAD^\kappa}(A;M)}$ then we have
	\begin{align*}
		\mathrm{(\mathfrak{D}^n_{mRBLD^\kappa}\circ \mathfrak{S}_n)(F_n,G_{n-1})}&=\mathrm{\mathfrak{D}^n_{mRBLD^\kappa}(\mathcal{S}_n(F_n),\mathcal{S}_{n-1}(G_{n-1}))}\\
		&=\mathrm{(\partial^n_{mRBLA^\kappa}\circ \mathcal{S}_n(F_n),\partial^{n-1}_{mRBLA^\kappa}\circ \mathcal{S}_{n-1}(G_{n-1})+(-1)^n\Delta^n\circ\mathcal{S}_n(F_n))}\\
		&=\mathrm{(\mathcal{S}_{n+1}\circ\partial^n_{mRBAA^\kappa}(F_n),\mathcal{S}_n\circ\partial^{n-1}_{mRBAA^\kappa}(G_{n-1})+(-1)^n\mathcal{S}_n\circ\Delta^n(F_n))}\\
		&=\mathrm{(\mathfrak{S}_{n+1}\circ\mathfrak{D}^n_{mRBAD^\kappa})(F_n,G_{n-1})}.
	\end{align*}
This complete the proof.
\end{proof}
\section{Deformations of modified Rota-Baxter AssDer pairs}\label{sec-4}
\def\theequation{\arabic{section}.\arabic{equation}}
\setcounter{equation} {0}
In this section, we study formal one-parameter deformations of modified Rota-Baxter AssDer pairs.
In particular, we show that the infinitesimal of a formal one-parameter deformation of a modified Rota-Baxter AssDer pair $\mathrm{(\mathcal{A},R,d)}$ is a $\mathrm{2}$-cocycle in the cohomology complex of $\mathrm{(\mathcal{A},R,d)}$ with
coefficients in itself.
Further, the vanishing of the second cohomology implies the rigidity of the modified Rota-Baxter AssDer pair.

Let $\mathrm{(\mathcal{A},R,d)}$ be a modified Rota-Baxter AssDer pair. Let $\mu$ denote the associative
multiplication on $\mathrm{A}$. Consider the space
$\mathrm{A[\![t]\!]}$ of formal power series in $\mathrm{t}$ with coefficients from $\mathrm{A}$. Then
$\mathrm{A[\![t]\!]}$ is a $\mathrm{\mathbf{k}[\![t]\!]}$-module.

\begin{defi}
A {\bf formal one-parameter deformation} of the modified Rota-Baxter AssDer pair $\mathrm{(\mathcal{A}=(A,\mu),R,d)}$
consists of three formal power series of the form
\begin{align*}
&\mathrm{\mu_t} =\mathrm{ \sum_{i=0}^\infty \mu_i t^i, \text{ where } \mu_i \in \mathrm{Hom}
(A^{\otimes 2}, A) \text{ with }
\mu_0 = \mu,}\\
&\mathrm{R_t} = \mathrm{\sum_{i=0}^\infty R_i t^i, \text{ where } R_i \in \mathrm{Hom} (A, A) \text{ with } R_0 = R},\\
&\mathrm{d_t} = \mathrm{\sum_{i=0}^\infty d_i t^i, \text{ where } d_i \in \mathrm{Hom} (A, A) \text{ with } d_0 = d},
\end{align*}
such that the $\mathrm{\mathbf{k}[\![t]\!]}$-module $\mathrm{A[\![t]\!]}$ is an associative algebra with the
$\mathrm{\mathbf{k}[\![t]\!]}$-bilinear
multiplication $\mathrm{\mu_t}$, the $\mathrm{\mathbf{k}[\![t]\!]}$-linear map
$\mathrm{R_t : A[\![t]\!] \rightarrow A[\![t]\!]}$ is a modified
Rota-Baxter operator of weight $\kappa$ and the $\mathrm{\mathbf{k}[\![t]\!]}$-linear map
$\mathrm{d_t : A[\![t]\!] \rightarrow A[\![t]\!]}$ is a derivation. In other words,
$\mathrm{\big( A[\![t]\!] = (A[\![t]\!], \mu_t), R_t,d_t \big)}$ is a
modified Rota-Baxter AssDer pair.

We often denote a formal one-parameter deformation as above by the triple $\mathrm{(\mu_t, R_t,d_t)}$.
\end{defi}

It follows that $\mathrm{(\mu_t, R_t,d_t)}$ is a formal one-parameter deformation of the modified Rota-Baxter AssDer pair
$\mathrm{(\mathcal{A}, R,d)}$ if and only if the followings equations hold:
\begin{align*}
\mathrm{\mu_t \big( \mu_t (a, b), c \big)} =~&\mathrm{ \mu_t \big( a, \mu_t(b,c)  \big)}, \\
\mathrm{\mu_t \big(  R_t (a), R_t (b) \big)} =~& \mathrm{R_t \big(  \mu_t (R_t (a), b) + \mu_t (a, R_t(b)) \big)
+ \kappa ~ \mu_t (a, b)},\\
\mathrm{ d_t \big(\mu_t(a,b)\big)} =~& \mathrm{ \mu_t (d_t (a), b) + \mu_t (a, d_t(b)) },\\
\mathrm{R_t \circ d_t(a)+R_t \circ d_t(a)}=~&\mathrm{d_t \circ R_t(a)+d_t\circ R_t(a)},
\end{align*}
for all $\mathrm{a,b,c \in A}$. By expanding these equations and comparing the coefficients of $\mathrm{t^n}$
(for $\mathrm{n \geq 0}$) in both sides, we obtain
\begin{align*}
\mathrm{\sum_{i+j=n} \mu_i \big( \mu_j (a, b), c  \big)} =~& \mathrm{\sum_{i+j = n} \mu_i \big(  a, \mu_j (b,c) \big)}, \\
\mathrm{\sum_{i+j+k = n} \mu_i \big( R_j (a), R_j (b)   \big)} =~& \mathrm{\sum_{i+j+k = n} R_i \big( \mu_j (R_k (a), b) ~+~ \mu_j
(a, R_k (b))  \big) + \kappa ~ \mu_n (a, b)},\\
\mathrm{\sum_{i+j = n} d_i \big( \mu_j (a,b) \big)} =~& \mathrm{\sum_{i+j = n} \mu_i (d_j (a), b) ~+~ \mu_i
(a, d_j (b))  \big) },\\
\mathrm{\sum _{\substack{\mathrm{i}+\mathrm{j}=n \\\mathrm{i},\mathrm{j}\geq 0}}R_\mathrm{i}\circ d_\mathrm{j}(a)}
=~&\mathrm{\sum _{\substack{\mathrm{i}+\mathrm{j}=n \\\mathrm{i},\mathrm{j}\geq 0}}d_\mathrm{i}\circ R_\mathrm{j}(a)},
\end{align*}
for $\mathrm{n \geq 0}$. Those equations hold for $\mathrm{n = 0}$ as $\mu$ is the associative multiplication
on $\mathrm{A}$,
the linear map $\mathrm{R: A \rightarrow A}$ is a modified Rota-Baxter algebra of weight $\kappa$
and the linear map $\mathrm{d: A \rightarrow A}$ is a derivation. However, for $\mathrm{n = 1}$, we get
\begin{align}
\mathrm{\mu_1 (a \cdot b, c) + \mu_1 (a,b) \cdot c} =~& \mathrm{\mu_1 (a, b \cdot c) + a \cdot \mu_1 (b, c)}, \label{inf-1-4}\\
\mathrm{R_1 (a) \cdot R(b) + R(a) \cdot R_1(b) + \mu_1 (R(a) , R(b))} =~& \mathrm{R_1 \big( R(a) \cdot b + a \cdot R(b)  \big)
+ R \big( \mu_1 (R(a), b) ~+~ \mu_1 (a, R(b))  \big)} \nonumber \\
~&\mathrm{+ R \big(  R_1(a) \cdot b + a \cdot R_1(b) \big) + \kappa ~ \mu_n (a, b)}, \label{inf-2-4}\\
\mathrm{d_1(\mu(a,b))+d(\mu_1(a,b))}=~&\mathrm{\mu_1(d(a),b)+\mu(d_1(a),b)+\mu_1(a,d(b))+\mu(a,d_1(b))},\label{inf-3-4}\\
\mathrm{R_1\circ d+R\circ d_1}=~&\mathrm{d_1\circ R+d\circ R_1}.\label{inf-4-4}
\end{align}
for all $\mathrm{a, b, c \in A}$.\\
Note that the equation (\ref{inf-1-4}) is equivalent to
 \begin{center}
 $\mathrm{(\delta_\mathrm{Hoch}^2 (\mu_1)) (a, b, c) = 0},$
 \end{center}
the equation (\ref{inf-2-4}) is equivalent to
\begin{center}
$\mathrm{-\delta_\mathrm{mHoch^\kappa}^1\mathrm{(R_1)(a,b)}-\phi^2(\mu_1)(\mathrm{a,b})=0},$
\end{center}
the equation (\ref{inf-3-4}) is equivalent to
\begin{center}
$\mathrm{\delta_\mathrm{Hoch}^1\mathrm{(d_1)(a,b)}+\Delta^2(\mu_1)(\mathrm{a,b})=0},$
\end{center}
while the equation (\ref{inf-4-4}) is equivalent to
\begin{center}
$\mathrm{\Delta^1(\mathrm{R_1(a)})-\phi^1(\mathrm{d_1(a)})=0}.$
\end{center}
Thus, we have,
$\mathrm{\Big((\delta_\mathrm{Hoch}^2(\mu_1),-\delta_\mathrm{mHoch^\kappa}^1\mathrm{(R_1)}-\phi^2\mu_1),
(\delta_\mathrm{Hoch}^1\mathrm{(d_1)}+\Delta^2\mu_1,\Delta^1\mathrm{R_1}-\phi^1\mathrm{d_1})\Big)=((0,0),(0,0))}$.\\
Which is equivalent to $\mathrm{(\partial_{mRBAA^\kappa}^2(\mu_1,R_1),\partial_{mRBAA^\kappa}^1(d_1)+\Delta^2(\mu_1,R_1))=0}$.\\
Hence, $\mathrm{\mathfrak{D}^2_\mathrm{mRBAD^\kappa}(\mu_1,R_1,d_1)=0}$.\\
This shows $\mathrm{(\mu_1,\mathrm{R_1,d_1})}$ is a $\mathrm{2}$-cocycle in the cochain complex
$\mathrm{\Big(\mathfrak{C}^\star_{mRBAD^\kappa}(A,A),\mathfrak{D}^\star_\mathrm{mRBAD^\kappa}\Big)}.$
Thus, from the above discussion, we have the following theorem.
 \begin{thm}\label{infy-co}
Let $\mathrm{(\mu_\mathrm{t}, \mathrm{R_t},\mathrm{d_t})}$ be a one-parameter formal deformation of a modified
Rota-Baxter AssDer pair $\mathrm{(\mathcal{A},R,d)}$. Then $\mathrm{(\mu_1,R_1,d_1)}$ is a
$\mathrm{2}$-cocycle in the cochain complex $\mathrm{\Big(\mathfrak{C}^\star_{mRBAD^\kappa}(A,A),
\mathfrak{D}^\star_\mathrm{mRBAD^\kappa}\Big)}.$
\end{thm}
\begin{defi}
The $\mathrm{2}$-cocycle $\mathrm{(\mu_1,R_1,d_1)}$ is called the {\bf infinitesimal} formal one-parameter deformation $(\mathrm{\mu_t,R_t,d_t)}$ of the modified Rota-Baxter AssDer pair $\mathrm{(\mathcal{A},R,d)}$.
\end{defi}

\begin{defi}
Let $\mathrm{(\mu_t, R_t,d_t)}$ and $\mathrm{(\mu_t', R_t',d_t')}$ be two formal one-parameter deformations of the
modified Rota-Baxter AssDer pair
$\mathrm{(\mathcal{A}, R,d)}$. These two deformations are called {\bf equivalent} if there exists a formal isomorphism
\begin{align*}
\mathrm{\varphi_t = \sum_{i=0}^\infty \varphi_i t^i : \big( A[\![t]\!] = (A[\![t]\!], \mu_t), R_t,d_t  \big)
\rightarrow  \big( A[\![t]\!]' = (A[\![t]\!], \mu_t'), R_t',d_t'  \big) ~~ \text{ with } \varphi_0 = Id}
\end{align*}
between modified Rota-Baxter AssDer pair, i.e.
\begin{align}\label{equiv-dend}
\mathrm{\varphi_t \big(  \mu_t (a, b) \big) = \mu_t' \big(\varphi_t (a), \varphi_t(b) \big), ~~~~ \varphi_t \circ R_t = R_t' \circ \varphi_t,
\text{ and } ~~~~  \varphi_t \circ d_t = d_t' \circ \varphi_t, ~ \text{ for } a, b \in A}.
\end{align}
\end{defi}

By expanding the previous equations in (\ref{equiv-dend}) and comparing the coefficients of $\mathrm{t^n}$
(for $\mathrm {n \geq 0}$) in both sides,
we obtain
\begin{align*}
\mathrm{\sum_{i+j=n} \varphi_i \big(  \mu_j (a, b) \big)} =~& \mathrm{\sum_{i+j +k = n} \mu_i' \big(  \varphi_j (a),
\varphi_k (b) \big)},\\
\mathrm{\sum_{i+j = n} \varphi_i \circ R_j} =~& \mathrm{\sum_{i+j = n} R_i' \circ \varphi_j},\\
\mathrm{\sum_{i+j = n} \varphi_i \circ d_j} =~& \mathrm{\sum_{i+j = n} d_i' \circ \varphi_j},
\end{align*}
for $\mathrm{n \geq 0}$. The equations are hold for $\mathrm{n=0}$ as $\mathrm{\varphi_0 = Id}$. However,
for $\mathrm{n=1}$, we get
\begin{align}
\mathrm{\mu_1 (a, b) + \varphi_1 (a \cdot b)} =~& \mathrm{\mu_1' (a, b) + \varphi_1 (a) \cdot b + a \cdot \varphi_1 (b)},
\label{mor-equiv-1}\\
\mathrm{R_1 + \varphi_1 \circ R} =~& \mathrm{R_1' +  R \circ \varphi_1,}  \label{mor-equiv-2}\\
\mathrm{d_1 + \varphi_1 \circ d} =~& \mathrm{d_1' +  d \circ \varphi_1,}  \label{mor-equiv-3}
\end{align}
Therefore, we have
\[\mathrm{(\mu_1^\prime,R_1^\prime,d_1^\prime)-(\mu_1,R_1,d_1)
=(\delta_{\mathrm{Hoch}}^1(\psi_1),-\phi^1(\psi_1),-\Delta^1(\psi_1))=\mathfrak{D}^1_\mathrm{mRBAD^\kappa}(\psi_1)
\in \mathfrak{C}^{1}_\mathrm{mRBAD^\kappa}(A,A)}.\]
As a consequence of the above discussions, we obtain the following.

\begin{thm}
Let $\mathrm{(\mathcal{A}, R,d)}$ be a modified Rota-Baxter AssDer pair. Suppose $\mathrm{(\mu_t, R_t,d_t)}$ is a formal
one-parameter deformation of $\mathrm{(\mathcal{A}, R,d)}.$ Then the infinitesimal is a $\mathrm{2}$-cocycle in the cohomology
complex of $\mathrm{(\mathcal{A}, R,d)}$ with coefficients in itself. Moreover, the corresponding cohomology class depends only
on the equivalence class of the deformation.
\end{thm}

We end this section by considering the rigidity of a modified Rota-Baxter AssDer pair.
We also find a sufficient condition for rigidity.

\begin{defi}
A modified Rota-Baxter AssDer pair $\mathrm{(\mathcal{A}, R,d)}$ is called {\bf rigid} if any formal one-parameter deformation
$\mathrm{(\mu_t, R_t,d_t)}$ is equivalent to the undeformed one $\mathrm{(\mu, R,d)}$.
\end{defi}

\begin{thm}
Let $\mathrm{(\mathcal{A}, R,d)}$ be a modified Rota-Baxter AssDer pair. If
$\mathrm{\mathcal{H}^2_\mathrm{mRBAD^\kappa} (A,A) = 0}$ then $\mathrm{(\mathcal{A},R,d)}$ is rigid.
\end{thm}

\begin{proof}
Let $\mathrm{(\mu_t, R_t,d_t)}$ be a formal one-parameter deformation of the modified Rota-Baxter AssDer pair
$\mathrm{(A,\mu,R,d)}$. From Theorem \ref{infy-co}, $\mathrm{(\mu_1,R_1,d_1)}$ is a $\mathrm{2}$-cocycle
and as $\mathrm{\mathcal{H}^2_\mathrm{mRBAD^\kappa}(A,A)=0}$, thus, there exists a $1$-cochain
$\psi_1\in\mathfrak{C}^1_\mathrm{mRBAD^\kappa}$ such that
\begin{equation}\label{eqt deformation}
	\mathrm{(\mu_1,R_1,d_1)=-\mathfrak{D}^1_\mathrm{mRBAD^\kappa}(\psi_1)}.
\end{equation}
Then setting $\mathrm{\psi_t=Id + \psi_1t}$, we have a deformation $(\bar{\mu}_t,\bar{R}_t,\bar{d}_t))$, where
\begin{eqnarray*}
\mathrm{\bar{\mu}_t(a,b)}&=&\mathrm{\big(\psi_t^{-1} \circ \mu_t \circ (\psi_t \circ \psi_t)\big)(a,b)},\\
\mathrm{\bar{R}_t(a)}&=&\mathrm{\big(\psi_t^{-1} \circ R_t \circ \psi_t\big)(a)},\\
\mathrm{ \bar{d}_t(a)}&=&\mathrm{\big(\psi_t^{-1} \circ d_t \circ \psi_t\big)(a)}.
\end{eqnarray*}
Thus, $\mathrm{(\bar{\mu}_t,\bar{R}_t,\bar{d}_t)}$ is equivalent to $\mathrm{(\mu_t,R_t,d_t)}$.\\
Moreover, we have
\begin{eqnarray*}
	\mathrm{\bar{\mu}_t(a,b)}&=& \mathrm{Id - \psi_1t+\psi^2t^2+\cdots+(-1)^i\psi_1^it^i+\cdots)
(\mu_t(a+\psi_1(a)t,y+\psi_1(b)t)},  \\
		\mathrm{\bar{R}_t(a)}&=& \mathrm{Id - \psi_1t+\psi^2t^2+\cdots+(-1)^i{\psi_1}^{i}t^i+
			\cdots) (R_t(a+\psi_1(a)t))},	\\
			\mathrm{\bar{d}_t(a)}&=& \mathrm{Id - \psi_1t+\psi^2t^2+\cdots+(-1)^i{\psi_1}^{i}t^i+
			\cdots) (d_t(a+\psi_1(a)t))}.
\end{eqnarray*}
Then,
\begin{eqnarray*}
\mathrm{\bar{\mu}_t(a,b)}&=&\mathrm{\mu(a,b)+\big(\mu_1(a,b)+\mu(a,\psi_1(b))+\mu(\psi_1(\psi_1(a),b)
-\psi_1(\mu(a,b)\big) +\bar{\mu}_2(a,b)t^2+\cdots},\\
\mathrm{\bar{R}_t(a)}&=&\mathrm{R(a)+\big(R(\psi_1(a))+R_1(a)-\psi_1(R(a))\big)t+\bar{R}_2(a)t^2+\cdots,} \\
\mathrm{\bar{d}_t(a)}&=&\mathrm{d(a)+\big(d(\psi_1(a))+d_1(a)-\psi_1(d(x))\big)t+\bar{d}_2(a)t^2+\cdots}.
\end{eqnarray*}
By \eqref{eqt deformation}, we have
\begin{eqnarray*}
	\mathrm{\bar{\mu}_t(a,b)}&=& \mathrm{\mu(a,b) +\bar{\mu}_2(a,b)t^2+\cdots} ,  \\
	\mathrm{\bar{R}_t(a)}&=&\mathrm{ R(a) +\bar{R}_2(a)t^2+\cdots},\\
	\mathrm{\bar{d}_t(a)}&=& \mathrm{d +\bar{d}_2(a)t^2+\cdots} .
\end{eqnarray*}

 Finally, by repeating
 the arguments, we can
 show that $\mathrm{(\mu_t,R_t, d_t)}$ is equivalent to the trivial deformation. Hence, $\mathrm{(A,R,d)}$ is rigid.
\end{proof}
	\section{Abelian extension of a modified Rota-Baxter AssDer pair}\label{sec5}
\def\theequation{\arabic{section}.\arabic{equation}}
\setcounter{equation} {0}
In this section, we study abelian extensions of the modified Rota-Baxter AssDer pair and show
that they are classified by the second cohomology. \\
Let $\mathrm{(\mathcal{A},R,d)}$ be a modified Rota-Baxter AssDer pair. Let  $\mathrm{(\mathcal{M},R_M,d_M)}$ be a triple consisting of a vector space $\mathrm{M}$ and linear maps $\mathrm{R_M:M\rightarrow M}$, $\mathrm{d_M:M\rightarrow M}$. One can see the triple $\mathrm{(\mathcal{M},R_M,d_M)}$ is a modified Rota-Baxter AssDer pair, where $\mathrm{M=(M,\mu_M)}$ is a trivial associative algebra.
\begin{defi}
	Let $\mathrm{(A,\mu,R,d)}$ be a modified Rota-Baxter AssDer pair and $\mathrm{M}$ be a vector space. A modified
	Rota-Baxter AssDer pair $\mathrm{(\hat{A},\mu_{\wedge},{\hat{R}},\hat{d})}$ is called an extension of
	$\mathrm{(A,\mu,R,d)}$ by $\mathrm{(M,\mu_\mathrm{M},R_M,d_M)}$ if there exists a short exact sequence 
	of morphisms of modified Rota-Baxter AssDer pair
	$$\begin{CD}
		0@>>> \mathrm{(\mathcal{M},d_M)} @>\mathrm{i} >> \mathrm{(\hat{\mathcal{A}},\hat{d})} @>\mathrm{p} >> \mathrm{(\mathcal{A},d)} @>>>\mathrm{0}\\
		@. @V {\mathrm{R_M}} VV @V \hat{\mathrm{R}} VV @V \mathrm{R} VV @.\\
		0@>>> \mathrm{(\mathcal{M},d_M)} @>\mathrm{i} >> \mathrm{(\hat{\mathcal{A}},\hat{d})} @>\mathrm{p} >> \mathrm{(\mathcal{A},d)} @>>>0
	\end{CD}$$
	where $\mathrm{\mu_M (m,n)=0}$ for all $\mathrm{m,n \in M}$ and $\mathrm{\hat{\mathcal{A}}=(\hat{A},\mu_\wedge)}$.
\end{defi}
An extension   $\mathrm{(\hat{A},\mu_{\wedge},{\hat{R}},\hat{d})}$ of the modified Rota-Baxter AssDer pair
$\mathrm{(A,\mu,R,d)}$ by $\mathrm{(M,\mu_\mathrm{M},R_M,d_M)}$ is called abelian if the associative algebra $\mathcal{M}$ is trivial.\\
A section of an abelian extension $\mathrm{(\hat{A},\mu_{\wedge},{\hat{R}},\hat{d})}$ of the modified Rota-Baxter AssDer pair
$\mathrm{(A,\mu,R,d)}$ by $\mathrm{(M,\mu_\mathrm{M},R_M,d_M)}$ consists of a linear map $\mathrm{s:A\rightarrow \hat{A}}$ such that $\mathrm{p\circ s=\mathrm{Id}}$. 
In the following, we always assume that $\mathrm{(\hat{\mathcal{A}},\hat{R},\hat{d})}$ is an abelian extension of the modified Rota-Baxter AssDer pair
$\mathrm{(A,\mu,R,d)}$ by $\mathrm{(M,\mu_\mathrm{M},R_M,d_M)}$ and $\mathrm{s}$ is a section of it.
For all $\mathrm{a\in A}$, $\mathrm{m\in M}$ define two bilinear maps
\begin{center}
	$\mathrm{\mathbf{l}:A\times M\rightarrow M}$ and $\mathrm{\mathbf{r}:M\times A\rightarrow M}$ respectively by
\end{center}
  \begin{equation*}
  	\mathrm{\mathbf{l}(a,m)=\mu_\wedge(s(a),m)}, \quad \mathrm{\mathbf{r}(m,a)=\mu_\wedge(m,s(a))}.
  \end{equation*}

\begin{pro}
	With the above notations, $\mathrm{(M,\mathbf{l},\mathbf{r},R_M,d_M)}$ is a representation of the modified Rota-Baxter AssDer 
	pair $\mathrm{(\mathcal{A},R,d)}$.
\end{pro}
\begin{proof}
	From \cite{Das2} we have $(\mathrm{M,\mathbf{l,r},R_M})$ is a bimodule over $\mathrm{(\mathcal{A},R)}$. Let $a\in A$ and $m\in M$ then we have $\mathrm{p(s(da)-\hat{d}(s(a)))=da-d(p(s(a)))=0}$ which means that 
	\begin{equation*}
		\mathrm{s(da)-\hat{d}(s(a))\in M}.
	\end{equation*}
Then 
\begin{align*}
	&\mathrm{\mu_\wedge(s(da)-\hat{d}(s(a)),m)=0}\\
	&\mathrm{\mu_\wedge(s(da),m)-\mu_\wedge(\hat{d}s(a),m)=0}\\
	&\mathrm{\mu_\wedge(s(da),m)-d_M(\mu_\wedge(s(a),m))+\mu_\wedge(s(a),d_M(m))=0}\\
	&\mathrm{\mathbf{l}(da,m)-d_M(\mathbf{l}(a,m))+\mathbf{l}(a,d_M(m))=0}.
\end{align*}
Those $\mathrm{d_M(\mathbf{l}(a,m))=\mathbf{l}(da,m)+\mathbf{l}(a,d_M(m))}$\\
Similarly we obtain $\mathrm{d_M(\mathbf{r}(m,a))=\mathbf{r}(d_M(m),a)+\mathbf{r}(m,da)}$. This complete the proof.
\end{proof}
For any $\mathrm{a,b\in A}$ and $\mathrm{m\in M}$, define
$\mathrm{\Theta:\otimes^2A\rightarrow M}$, $\mathrm{\chi:A\rightarrow M}$ and $\mathrm{\xi:A\rightarrow M}$ as follows
\begin{eqnarray*}
	\mathrm{\Theta(a,b)}&=&\mathrm{\mu_\wedge(s(a),s(b))-s(\mu (a,b))},\\
	\mathrm{\chi(a)}&=&\mathrm{\hat{d}(s(a))-s(d(a))},\\
	\mathrm{\xi(a)}&=&\mathrm{\hat{R}(s(a))-s(R(a)),\quad \forall a,b\in A}.
\end{eqnarray*}
These linear maps lead to define \\
$\mathrm{R_\xi:A\oplus M\rightarrow A\oplus M}$ and $\mathrm{d_\chi:A\oplus M\rightarrow A\oplus M}$ by
\begin{eqnarray*}
	\mathrm{R_\xi(a+m)}&=&\mathrm{R(a)+R_M(m)+\xi(a)},\\
	\mathrm{d_\chi(a)}&=&\mathrm{d(a)+d_M(m)+\chi(a)}.
\end{eqnarray*}
\begin{thm}\label{theorem ext}
	With the above notations, the quadruple $\mathrm{(A\oplus M,\mu_\Theta,R_\xi,d_\chi)}$ where
	\begin{equation*}
		\mathrm{\mu_\Theta(a+u,b+v)=\mu(a,b)+\Theta(a,b),\quad \forall a,b\in A,\quad \forall m,n\in M},
	\end{equation*}
	is a modified Rota-Baxter AssDer pair if and only if $\mathrm{(\Theta,\xi,\chi)}$ is a $\mathrm{2}$-cocycle of the modified Rota-Baxter AssDer pair $\mathrm{(\mathcal{A},R,d)}$ with coefficients in the trivial $\mathcal{A}$-bimodule.
\end{thm}
\begin{proof}
	If $\mathrm{(A\oplus M,\mu_\Theta,R_\xi,d_\chi)}$ is a modified Rota-Baxter AssDer pair, it means that $\mathrm{(\mathcal{A},\mu_\Theta)}$ is an associative algebra, $\mathrm{R_\xi}$ is a modified Rota-Baxter operator, $\mathrm{d_\chi}$ is a derivation and $\mathrm{R_\xi\circ d_\chi=d_\chi\circ R_\xi}$.\\
	The couple $\mathrm{(\mathcal{A},\mu_\Theta)}$ is an associative algebra means that
	\begin{equation*}
		\mathrm{\mu_\Theta(\mu_\Theta(a+m,b+n),c+q)-\mu_\Theta(a+m,\mu_\Theta(b+n,c+q))=0}.
	\end{equation*}
	Which is exactly
	\begin{equation}
		\mathrm{\delta^2_\mathrm{Hoch}\Theta=0}.\label{ext1}
	\end{equation}
	And $\mathrm{R_\xi}$ is a modified Rota-Baxter operator on $\mathrm{(\mathcal{A},\mu_\Theta)}$ means that
	\begin{equation*}
		\mathrm{\mu_\Theta(R_\xi(a+m),R_\xi(b+n))=R_\xi\Big(\mu_\Theta(R_\xi(a+u),b+v)+\mu_\Theta(a+u,R_\xi(b+v))\Big)
			+\kappa \mu_\Theta(a+u,b+v)}.
	\end{equation*}
	Since $\mathrm{(\mathcal{A},R)}$ is a modified Rota-Baxter Lie algebra and using equation 
	\eqref{rep mRBAD3} we get
	\begin{align*}
		&\mathrm{\mu_\theta(R_\xi(a+m),R_\xi(b+n))-R_\xi\Big(\mu_\theta(R_\xi(a+m),b+n) + \mu_\theta(a+m,R_\xi(b+n))\Big)-\kappa \mu_\theta(a+m,b+n)}\\
		=&\mathrm{\mu_\theta(R(a)+R_M(m)+\xi(R(a)),R(b)+R_M(n)+\xi(R(b))) -R_\xi(\mu_\theta(R(a)+R_M(m)+\xi(R(a)),b+n))}\\
		&-\mathrm{R_\xi( \mu_\theta(a+m,R(b)+R_M(n)+\xi(R(b))))-\kappa \mu_\theta(a+m,b+n)}\\
		=&\mathrm{\mu(R(a),R(b))+\Theta(R(a),R(b))-R_\xi(\mu(R(a),b) +\Theta(R(a),b))-R_\xi(\mu(a,R(b)) +\Theta(a,R(b)))-\kappa\mu(a,b) -\kappa\Theta(a,b)}\\
		=&\mathrm{\Theta(R(a),R(b))-R_M\Theta(R(a),b)-R_M(a,R(b))-\kappa \Theta(a,b)-\xi(\mu(R(a),b) +\mu(a,R(b)) )}.
	\end{align*}
	Then $\mathrm{R_\xi}$ is a modified Rota-Baxter operator on $\mathrm{(\mathcal{A},\mu_\theta)}$ if and only if
	\begin{equation*}
		\mathrm{\Theta(R(a),R(b))-R_M\Theta(R(a),b)-R_M(a,R(b))-\kappa \Theta(a,b)-\xi(\mu(R(a),b) +\mu(a,R(b)) )}.
	\end{equation*}
	Which is exactly
	\begin{equation}\label{ext2}
		\mathrm{-\delta^1_\mathrm{mHoch}(\xi)+\Delta^2(\Theta)=0}.
	\end{equation}
	And $\mathrm{d_\chi}$ is a derivation on the associative algebra $\mathrm{(\mathcal{A},\mu_\theta)}$ if and only if
	\begin{equation*}
		\mathrm{d_\chi \mu_\theta(a+m,b+n)=\mu_\theta(d_\chi(a+m),b+n) + \mu_\theta(a+m,d_\chi(b+n))}.
	\end{equation*}
	Using the fact that $\mathrm{d}$ is a derivation on $\mathrm{\mathcal{A}}$ we obtain
	\begin{align*}
		&\mathrm{d_\chi(\mu_\theta(a+m,b+n) )- \mu_\theta(d_\chi(a+m),b+n) -\mu_\theta(a+m,d_\chi(b+n))}\\
		=&\mathrm{d_\chi(\mu_\theta(a+m,b+n) )- \mu_\theta(d(a)+d_M(m)+\chi(a),b+n) - \mu_\theta(a+m,d(b)+d_M(n)+\chi(b))}\\
		=&\mathrm{d_\chi(\mu(a,b) +\Theta(a,b))- \mu(d(a),b) -\Theta(d(a),b)-\mu(a,d(b)) -\Theta(a,d(b))}\\
		=&\mathrm{d \mu(a,b) +d_M(\Theta(a,b))+\chi(\mu(a,b))- \mu(d(a),b) -\Theta(d(a),b)- \mu(a,d(b)) -\Theta(a,d(b))}\\
		=&\mathrm{d_M(\Theta(a,b))-\Theta(d(a),b)-\Theta(a,d(b))+\chi(\mu(a,b))}.
	\end{align*}
	Then $\mathrm{d_\chi}$ is a derivation on the associative algebra $\mathrm{(\mathcal{A},\mu_\theta)}$ if and only if
	\begin{equation*}
		\mathrm{d_M(\Theta(a,b))-\Theta(d(a),b)-\Theta(a,d(b))+\chi(\mu(a,b))}.
	\end{equation*}
	Which is exactly
	\begin{equation}
		\mathrm{\delta^1_\mathrm{Hoch}(\chi)+\Delta^2(\Theta)=0}.\label{ext3}
	\end{equation}
	Finally, using the equations \eqref{equation RBAD1} and \eqref{rep mRBAD3} we get
	\begin{align*}
		&\mathrm{d_\chi\circ R_\xi(a+m)-d_\xi\circ d_\chi(a+m)}\\
		=&\mathrm{d_\chi(R(a)+R_M(m)\xi(a))-R_\xi(d(a)+d_M(m)+\chi(a))}\\
		=&\mathrm{d(R(a))+d_M(R_M(m)+\xi(a))+\chi(R(a))-R(d(a))-R_M(d_M(m)+\chi(a))-\xi(d(a))}\\
		=&\mathrm{d_M(\xi(a))-\xi(d(a))+\chi(R(a))-R_M(\chi(a))}
	\end{align*}
	then $\mathrm{R_\xi}$ and $\mathrm{d_\chi}$ commute if and only if
	\begin{equation*}
		\mathrm{d_M(\xi(a))-\xi(d(a))+\chi(R(a))-R_M(\chi(a))=0,\quad \forall a\in A}.
	\end{equation*}
	Which is exactly
	\begin{equation}\label{ext4}
		\mathrm{\Delta^1(\xi)-\phi^1(\chi)=0}.
	\end{equation}
	In conclusion, $\mathrm{(A\oplus M,\mu_\Theta,R_\xi,d_\chi)}$ is a modified Rota-Baxter AssDer pair if and 
	only if equations \eqref{ext1}, \eqref{ext2}, \eqref{ext3}, \eqref{ext4} hold.\\
	For the second sense, if $\mathrm{(\Theta,\xi,\chi)\in\mathfrak{C}^2_\mathrm{mRBAD^\kappa}(A;V)}$ is a 
	$\mathrm{2}$-cocycle if and only if
	\begin{equation*}
		\mathrm{(\partial_{mRBAA^\kappa}^2(\theta,\xi),\partial_{mRBAA^\kappa}^1(\chi)+\Delta^2(\theta,\xi))=0},
	\end{equation*}
which is equivalent to
	\begin{equation*}
		\mathrm{\Big(\delta^2_\mathrm{Hoch}(\Theta),-\delta^1_\mathrm{mHoch}(\xi)-\phi^2(\Theta),  
			\delta^1_\mathrm{Hoch}(\chi)+\Delta^2(\Theta),\Delta^1(\xi)-\phi^1(\chi)\Big)=0}.
	\end{equation*}
	This means that equations \eqref{ext1}, \eqref{ext2}, \eqref{ext3}, \eqref{ext4} are satisfied. \\
	
	Thus $\mathrm{\Big(\delta^2_\mathrm{Hoch}(\Theta),-\delta^1_\mathrm{mHoch}(\xi)-\phi^2(\Theta),  
		\delta^1_\mathrm{Hoch}(\chi)+\Delta^2(\Theta),\Delta^1(\xi)-\phi^1(\chi)\Big)=0}$
	if and only if $\mathrm{(A\oplus M,\mu_\Theta)}$ is an associative algebra, $\mathrm{R_\xi}$ is a modified 
	Rota-Baxter operator of  $\mathrm{(A\oplus M,\mu_\Theta)}$, $\mathrm{d_\chi}$ is a derivation of 
	$\mathrm{(A\oplus M,\mu_\Theta)}$ and $\mathrm{d_\chi\circ R_\xi=R_\xi\circ d_\chi}$. This complete the proof.

\end{proof}
\begin{defi}
	Two abelian extensions $\mathrm{(\hat{A}_1,\mu_{\wedge_1},{\hat{R}_1},\hat{d}_1)}$ and 
	$\mathrm{(\hat{A}_2,\mu_{\wedge_2},{\hat{R}_2},\hat{d}_2)}$ of a modified Rota-Baxter AssDer pair 
	$\mathrm{(A,\mu,R,d)}$ by $\mathrm{(M,\mu_\mathrm{M},R_M,d_M)}$ are called equivalent if there exists a 
	homomorphism of modified Rota-Baxter AssDer pairs 
	$\mathrm{\gamma: (\hat{A}_1,\mu_{\wedge_1},{\hat{R}_1},\hat{d}_1) \rightarrow  
		(\hat{A}_2,\mu_{\wedge_2},{\hat{R}_2},\hat{d}_2)}$ such that the following diagram commutes
	
	$$\begin{CD}
		0@>>> {\mathrm{(M,\mathrm{R}_M,\mathrm{d_M})}} @>\mathrm{i_1} >> \mathrm{(\hat{A}_1,\hat {\mathrm{R}}_1,\hat{d}_1)} @>\mathrm{p_1} >> \mathrm{(A,\mathrm{R},\mathrm{d})} @>>>\mathrm{0}\\
		@. @| @V \mathrm{\gamma} VV @| @.\\
		0@>>> {\mathrm{(M,\mathrm{R}_M,\mathrm{d_M})}} @>\mathrm{i_2} >> \mathrm{(\hat{A}_2,\hat {\mathrm{R}}_2,\hat{d}_2)} @>\mathrm{p_2} >> \mathrm{(A,\mathrm{R},\mathrm{d})} @>>>\mathrm{0}.
	\end{CD}$$
\end{defi}
\begin{thm}
	Abelian extensions of modified Rota-Baxter AssDer pair $\mathrm{(A,\mu,R,d)}$ by $\mathrm{(M,\mu_\mathrm{M},R_M,d_M)}$
	are classified by the second cohomology $\mathrm{\mathcal{H}^2_\mathrm{mRBAD^\kappa}(A;M)}$ of the modified 
	Rota-Baxter AssDer pair $\mathrm{(A,\mu,R,d)}$ with coefficients in the trivial bimodule.
\end{thm}
\begin{proof}
	Let $\mathrm{(\hat{A},\mu_\wedge,\hat{R},\hat{d})}$ be an abelian extension of a modified Rota-Baxter AssDer 
	pair $\mathrm{(A,\mu,R,d)}$ by $\mathrm{(M,\mu_\mathrm{M},R_M,d_M)}$. Let $\mathrm{s}$ be a section of it where 
	$\mathrm{s:A\rightarrow\hat{A}}$, we already have a $\mathrm{2}$-cocycle $\mathrm{(\Theta,\xi,\chi)\in\mathfrak{C}^2_\mathrm{mRBAD^\kappa}(A;M)}$ 
	by theorem \eqref{theorem ext}.\\
	First, we prove that the cohomological class of $\mathrm{(\Theta,\xi,\chi)}$ does not depend on the choice of 
	sections. Assume that $\mathrm{(s_1,s_2)}$ are two different sections providing $\mathrm{2}$-cocycles $\mathrm{(\Theta_1,\xi_1,\chi_1)}$ 
	and $\mathrm{(\Theta_2,\xi_2,\chi_2)}$ respectively. Define a linear map $\mathrm{\mathfrak{h}:A\rightarrow M}$ by $\mathrm{\mathfrak{h}(a)=s_1(a)-s_2(a),\quad \forall a\in A}$. 
	Then
	\begin{align*}
		\mathrm{\Theta_1(a,b)}&=\mathrm{\mu_\wedge(s_1(a),s_1(b)) -s_1(\mu(a,b))}\\
		&\mathrm{=\mu_\wedge(\mathfrak{h}(a)+s_2(a),\mathfrak{h}(a)+s_2(a) )-\mathfrak{h}(\mu(a,b))-s_2(\mu(a,b))}\\
		&\mathrm{=\mu_\wedge(s_2(a),s_2(b)) -\mathfrak{h}(\mu(a,b))}\\
		&\mathrm{=\Theta_2(a,b)+\delta^1_\mathrm{CE}(\mathfrak{h})(a,b)}
	\end{align*}
	And
	\begin{align*}
		\mathrm{\xi_1(a)}&\mathrm{=\hat{R}(s_1(a))-s_1(R(a))}\\
		&\mathrm{=\hat{R}(\mathfrak{h}(a)+s_2(a))-\mathfrak{h}((R(a)))-s_2(R(a))}\\
		&\mathrm{=\xi_2(a)+R_V(\mathfrak{h}(a))-\mathfrak{h}(R(a))}\\
		&\mathrm{=\xi_2(a)-\phi^1\mathfrak{h}(a)}.
	\end{align*}
	Also
	\begin{align*}
		\mathrm{\chi_1(a)}&=\mathrm{\hat{d}(s_1(a))-s_1(d(a))}\\
		&\mathrm{=\hat{d}(\mathfrak{h}(a)+s_2(a))-\mathfrak{h}((d(a)))-s_2(d(a))}\\
		&\mathrm{=\chi_2(a)+d_V(\mathfrak{h}(a))-\mathfrak{h}(d(a))}\\
		&\mathrm{=\xi_2(a)-\Delta^1\mathfrak{h}(a)}.
	\end{align*}
	Which means that
	\begin{equation*}
		\mathrm{(\Theta_1,\xi_1,\chi_1)=(\Theta_2,\xi_2,\chi_2)+\mathfrak{D}^1_\mathrm{mRBLD^\kappa}(\mathfrak{h})}.
	\end{equation*}
	So $\mathrm{(\Theta_1,\xi_1,\chi_1)}$ and $\mathrm{(\Theta_2,\xi_2,\chi_2)}$ are in the same cohomological class.\\
	Next, we show that equivalent abelian extensions gives the same element in $\mathrm{\mathcal{H}^2_\mathrm{mRBAD^\kappa}(A;M)}$. Let $\mathrm{(\hat{A}_1,\mu_{\wedge_1},\hat{R}_1,\hat{d}_1)}$ and $\mathrm{(\hat{A}_2,\mu_{\wedge_2},\hat{R}_2,\hat{d}_2)}$ 
	be two equivalent abelian extensions of a modified Rota-Baxter AssDer pair $\mathrm{(A,\mu,R,d)}$ by 
	$\mathrm{(M,\mu_\mathrm{M},R_M,d_M)}$ via the homomorphism $\mathrm{\gamma}$. Assume that $\mathrm{s_1}$ is a section of 
	$\mathrm{(\hat{A}_1,\mu_{\wedge_1},\hat{R}_1,\hat{d}_1)}$ and $\mathrm{(\Theta_1,\xi_1,\chi_1)}$ is the corresponding 
	$\mathrm{2}$-cocycle. Since $\mathrm{\gamma}$ being a homomorphism of modified Rota-Baxter AssDer pairs such that 
	$\mathrm{\gamma_{\vert M}=\mathrm{Id_M}}$, we have then
	\begin{align*}
		\mathrm{\xi_2(a)}&\mathrm{=\hat{R}_2(a)-s_2(R(a))}\\
		&\mathrm{=\hat{R}_2(\gamma(s_1(a)))-\gamma(s_1(R(a)))}\\
		&\mathrm{=\gamma(\hat{R}_1(s_1(a))-s_1(R(a)))}\\
		&\mathrm{=\xi_1(a)}
	\end{align*}
	and
	\begin{align*}
		\mathrm{\chi_2(a)}&\mathrm{=\hat{d}_2(a)-s_2(d(a))}\\
		&\mathrm{=\hat{d}_2(\gamma(s_1(a)))-\gamma(s_1(d(a)))}\\
		&\mathrm{=\gamma(\hat{d}_1(s_1(a))-s_1(d(a)))}\\
		&\mathrm{=\chi_1(a)}
	\end{align*}
	similarly we obtain $\mathrm{\Theta_2(a,b)=\Theta_1(a,b)}$. Thus, equivalent abelian extension gives the 
	same element in $\mathrm{\mathcal{H}^2_\mathrm{mRBAD^\kappa}(A;M)}$.\\
	On the other hand, given two $\mathrm{2}$-cocycles $\mathrm{(\Theta_1,\xi_1,\chi_1)}$ and 
	$\mathrm{(\Theta_2,\xi_2,\chi_2)}$, we have two abelian extensions 
	$\mathrm{(A\oplus M,\mu_{\Theta_1},R_{\xi_1},d_{\chi_1})}$ and 
	$\mathrm{(A\oplus M,\mu_{\Theta_2},R_{\xi_2},d_{\chi_2})}$ by theorem \eqref{theorem ext}. 
	Suppose that they belong to the same cohomological class in $\mathrm{\mathcal{H}^2_\mathrm{mRBAD^\kappa}(A;M)}$, 
	then the existence of a linear map $\mathrm{\mathfrak{h}:A\rightarrow M}$ such that
	\begin{equation*}
		\mathrm{(\Theta_1,\xi_1,\chi_1)=(\Theta_2,\xi_2,\chi_2)+\mathfrak{D}^1_\mathrm{mRBAD^\kappa}(\mathfrak{h})}.
	\end{equation*}
	Define $\mathrm{\gamma:A\oplus M\rightarrow A\oplus M}$ by $\mathrm{\gamma(a+m)=a+\mathfrak{h}(a)+m}$, for all $\mathrm{a\in A}$ and $\mathrm{m\in M}$.
	\begin{align*}
		\mathrm{\gamma(\mu_{\Theta_1}(a+m,b+n))- \mu_{\Theta_2}(\gamma(a+m),\gamma(b+n)) }
		&\mathrm{=\gamma(\mu(a,b)+\Theta_1(a,b))- \mu_{\Theta_2}(a+\mathfrak{h}(a)+m,b+\mathfrak{h}(b)+n) }\\
		&\mathrm{=\gamma(\mu(a,b)+\Theta_1(a,b))-\mu(a,b)-\Theta_2(a,b)}\\
		&\mathrm{=\mu(a,b)+\mathfrak{h}(\mu(a,b))+\Theta_1(a,b)-\mu(a,b)-\Theta_2(a,b)}\\
		&\mathrm{=\Theta_1(a,b)-\Theta_2(a,b)-\delta^1_\mathrm{CE}(\mathfrak{h})(a,b)}\\
		&\mathrm{=0}
	\end{align*}
	similarly we have $\mathrm{\gamma\circ R_{\xi_1}=R_{\xi_2}\circ\gamma}$ and 
	$\mathrm{\gamma\circ d_{\chi_1}=d_{\chi_2}\circ\gamma}$. Thus $\mathrm{\gamma}$ is a 
	homomorphism of these two abelian extensions, this complete the proof.
\end{proof}

\noindent {\bf Acknowledgment:}
The authors would like to thank the referee for valuable comments and suggestions on this article. Ripan Saha is supported by the Science and Engineering Research Board
(SERB), Department of Science and Technology (DST), Govt. of India (Grant Number- CRG/2022/005332).


\end{document}